\documentclass[12pt,leqno]{article}
\usepackage{amsmath,amssymb,bbm,array}
\usepackage{amsthm}
\usepackage{titlesec} % ???]???`???D????????
\usepackage[all,cmtip]{xy}

\titlespacing{\section}{0cm}{3.5pc}{1.5pc}

\makeatletter
\def\@citex[#1]#2{\if@filesw\immediate\write\@auxout{\string\citation{#2}}\fi
  \def\@citea{}\@cite{\@for\@citeb:=#2\do
    {\@citea\def\@citea{\@citesep}\@ifundefined
       {b@\@citeb}{{\bf ?}\@warning
       {Citation `\@citeb' on page \thepage \space undefined}}%
{\csname b@\@citeb\endcsname}}}{#1}}
\def\@citesep{; }
\makeatother

\newtheoremstyle{Kang}{}{}{\itshape}{}{\bf}{}{.5em}{}
\theoremstyle{Kang}
\newtheorem{theorem}{Theorem}[section]

\newtheorem{lemma}[theorem]{Lemma}

\newtheoremstyle{Kremark}{}{}{}{}{\bf}{}{.5em}{}
\theoremstyle{Kremark}
\newtheorem*{remark}{Remark.}
\newtheorem{defn}[theorem]{Definition}
\newtheorem{other}{}

\newtheorem{question}[theorem]{Question}

\newenvironment{idef}[1]{\renewcommand{\theother}{#1}\begin{other}}{\end{other}}
\newenvironment{Case}[1]{\medskip {\it Case #1.}}{}

\allowdisplaybreaks[1]  % ???\?h?????????

\def\fn#1{\operatorname{#1}} % function work like \sin
\def\bm#1{\mathbbm{#1}}

\def\side#1#2{\mathop{^{#1}\mkern-2mu#2}}
\title{NOETHER'S PROBLEM AND \\[2mm] UNRAMIFIED BRAUER GROUPS}
\author{Akinari Hoshi$^1$, Ming-chang Kang$^2$ and Boris E. Kunyavskii$^3$ \\[3mm]
\begin{minipage}{16cm} \begin{description} \itemsep=-1pt
\item[] $^{(1)}$Department of Mathematics, Rikkyo University, Tokyo, Japan
\item[] $^{(2)}$Department of Mathematics and Taida Institute of Mathematical\\ Sciences,
National Taiwan University, Taipei, Taiwan
\item[] $^{(3)}$Department of Mathematics, Bar-Ilan University, 52900 Ramat\\ Gan, Israel
\end{description} \end{minipage}}
\date{}

\begin{document}

\maketitle

\footnote{\textit{\!\!\!$2010$ Mathematics Subject
Classification}. Primary 13A50, 14E08, 14M20, 20J06, 12F12.}
\footnote{\textit{\!\!\!Keywords and phrases}. Noether's problem,
rationality problem, unramified Brauer groups, Bogomolov
multipliers, rationality, retract rationality.}
\footnote{\!\!\!E-mails: hoshi@rikkyo.ac.jp, kang@math.ntu.edu.tw,
kunyav@macs.biu.ac.il.} \footnote{\!\!\!The first-named author was
partially supported by KAKENHI (22740028). The second-named author
was partially supported by National Center for Theoretic Sciences
(Taipei Office). The third-named author was partially supported by
the Minerva Foundation through the Emmy Noether Research Institute
for Mathematics. Parts of the work of this paper were finished
while the first-named author visited National Center for Theoretic
Sciences (Taipei).}

\begin{abstract}
{\noindent\bf Abstract.} Let $k$ be any field, $G$ be a finite group
acing on the rational function field $k(x_g:g\in G)$ by $h\cdot
x_g=x_{hg}$ for any $h,g\in G$. Define $k(G)=k(x_g:g\in G)^G$.
Noether's problem asks whether $k(G)$ is rational (= purely
transcendental) over $k$. It is known that, if $\bm{C}(G)$ is
rational over $\bm{C}$, then $B_0(G)=0$ where $B_0(G)$ is the
unramified Brauer group of $\bm{C}(G)$ over $\bm{C}$. Bogomolov
showed that, if $G$ is a $p$-group of order $p^5$, then $B_0(G)=0$.
This result was disproved by Moravec for $p=3,5,7$ by computer
calculations. We will prove the following theorem. Theorem. Let $p$
be any odd prime number, $G$ be a group of order $p^5$. Then
$B_0(G)\ne 0$ if and only if $G$ belongs to the isoclinism family
$\Phi_{10}$ in R. James's classification of groups of order $p^5$.
\end{abstract}

\newpage
%---------------------------------S1
\section{Introduction}

Let $k$ be any field and $G$ be a finite group. Let $G$ act on the
rational function field $k(x_g:g\in G)$ by $k$-automorphisms so that
$g\cdot x_h=x_{gh}$ for any $g,h\in G$. Denote by $k(G)$ the fixed
field $k(x_g:g\in G)^G$. Noether's problem asks whether $k(G)$ is
rational (= purely transcendental) over $k$. It is related to the
inverse Galois problem, to the existence of generic $G$-Galois
extensions over $k$, and to the existence of versal $G$-torsors over
$k$-rational field extensions \cite[33.1, p.~86]{Sw,Sa1,GMS}.
Noether's problem for abelian groups was studied by Swan,
Voskresenskii, Endo, Miyata and Lenstra, etc. The reader is referred
to Swan's paper for a survey of this problem \cite{Sw}.

On the other hand, just a handful of results about Noether's problem
are obtained when the groups are not abelian. It is the case even
when $G$ is a $p$-group.

Before stating the results of Noether's problem for non-abelian
$p$-groups, we recall some relevant definitions.

%-------------------d1.1
\begin{defn} \label{d1.1}
Let $k\subset K$ be an extension of fields. $K$ is rational over
$k$ (for short, $k$-rational) if $K$ is purely transcendental over
$k$. $K$ is stably $k$-rational if $K(y_1,\ldots,y_m)$ is rational
over $k$ for some $y_1,\ldots,y_m$ such that $y_1,\ldots,y_m$ are
algebraically independent over $K$. When $k$ is an infinite field,
$K$ is said to be retract $k$-rational if there is a $k$-algebra
$A$ contained in $K$ such that (i) $K$ is the quotient field of
$A$, (ii) there exist a non-zero polynomial $f\in
k[X_1,\ldots,X_n]$ (where $k[X_1,\ldots,X_n]$ is the polynomial
ring) and $k$-algebra homomorphisms $\varphi\colon A\to
k[X_1,\ldots,X_n][1/f]$ and $\psi\colon k[X_1,\ldots,X_n][1/f]\to
A$ satisfying $\psi\circ\varphi =1_A$. (See \cite{Sa2,Ka} for
details.) It is not difficult to see that ``$k$-rational"
$\Rightarrow$ ``stably $k$-rational" $\Rightarrow$ ``retract
$k$-rational".
\end{defn}

%-------------------d1.2
\begin{defn} \label{d1.2}
Let $k\subset K$ be an extension of fields. The notion of the
unramified Brauer group of $K$ over $k$, denoted by
$\fn{Br}_{v,k}(K)$ was introduced by Saltman \cite{Sa3}. By
definition, $\fn{Br}_{v,k}(K)=\bigcap_R \fn{Image} \{ \fn{Br}(R)\to
\fn{Br}(K)\}$ where $\fn{Br}(R)\to \fn{Br}(K)$ is the natural map of
Brauer groups and $R$ runs over all the discrete valuation rings $R$
such that $k\subset R\subset K$ and $K$ is the quotient field of
$R$.
\end{defn}

%-------------------l1.3
\begin{lemma}[Saltman \cite{Sa3,Sa5}] \label{l1.3}
If $k$ is an infinite field and $K$ is retract $k$-rational, then
the natural map $\fn{Br}(k)\to \fn{Br}_{v,k} (K)$ is an isomorphism.
In particular, if $k$ is an algebraically closed field and $K$ is
retract $k$-rational, then $\fn{Br}_{v,k}(K)=0$.
\end{lemma}

%---------------------t1.4
\begin{theorem}[{Bogomolov, Saltman \cite[Theorem 12]{Bo,Sa4}}] \label{t1.4}
Let $G$ be a finite group, $k$ be an algebraically closed field with
$\gcd \{|G|,\fn{char}k\}=1$. Let $\mu$ denote the multiplicative
subgroup of all roots of unity in $k$. Then $\fn{Br}_{v,k}(k(G))$ is
isomorphic to the group $B_0(G)$ defined by
\[
B_0(G)=\bigcap_A \fn{Ker} \{\fn{res}_G^A: H^2(G,\mu)\to H^2(A,\mu)\}
\]
where $A$ runs over all the bicyclic subgroups of $G$ $($a group $A$
is called bicyclic if $A$ is either a cyclic group or a direct
product of two cyclic groups$)$.
\end{theorem}

Note that $B_0(G)$ is a subgroup of $H^2(G,\mu)$ (where $\gcd
\{|G|,\fn{char}k\}=1$). Since $H^2(G,\mu) \simeq H_2(G)$, which is
the Schur multiplier of $G$ (see \cite{Kar}), we will call
$B_0(G)$ the Bogomolov multiplier of $G$, following the convention
in \cite{Ku}. Because of Theorem \ref{t1.4} we will not
distinguish $B_0(G)$ and $\fn{Br}_{v,k}(k(G))$ when $k$ is
algebraically closed and $\gcd \{|G|,\fn{char}k\}=1$. In this
situation, $B_0(G)$ is canonically isomorphic to $\bigcap_A
\fn{Ker}\{ \fn{res}_G^A$ $\colon H^2(G,\bm{Q}/\bm{Z}) \to
H^2(A,\bm{Q}/\bm{Z})\}$, i.e.\ we may replace the coefficient
$\mu$ by $\bm{Q}/\bm{Z}$ in Theorem \ref{t1.4}.

Using the unramified Brauer groups, Saltman and Bogomolov are able
to establish counter-examples to Noether's problem for non-abelian
$p$-groups.

%-----------------------t1.5
\begin{theorem} \label{t1.5}
Let $p$ be any prime number, $k$ be any algebraically closed field
with $\fn{char}k\ne p$.

{\rm (1) (Saltman \cite{Sa3})} There is a group $G$ of order $p^9$
such that $B_0(G)\ne 0$. In particular, $k(G)$ is not retract
$k$-rational. Thus $k(G)$ is not $k$-rational.

{\rm (2) (Bogomolov \cite{Bo})} There is a group $G$ of order
$p^6$ such that $B_0(G)\ne 0$. Thus $k(G)$ is not $k$-rational.
\end{theorem}

For $p$-groups of small order, we have the following result.

%-------------------------t1.6
\begin{theorem}[Chu and Kang \cite{CK}] \label{t1.6}
Let $p$ be any prime number, $G$ is a $p$-group of order $\le p^4$
and of exponent $e$. If $k$ is a field satisfying either {\rm (i)}
$\fn{char}k=p$, or {\rm (ii)} $k$ contains a primitive $e$-th root
of unity, then $k(G)$ is $k$-rational.
\end{theorem}

Because of the above Theorems \ref{t1.5} and \ref{t1.6}, we may
wonder what happens to non-abelian $p$-groups of order $p^5$.

%-------------------------t1.7
\begin{theorem}[Chu, Hu, Kang and Prokhorov \cite{CHKP}] \label{t1.7}
Let $G$ be a group of order $32$ and of exponent $e$. If $k$ is a
field satisfying either {\rm (i)} $\fn{char}k=2$, or {\rm (ii)} $k$
contains a primitive $e$-th root of unity, then $k(G)$ is
$k$-rational. In particular, $B_0(G)=0$.
\end{theorem}

Working on $p$-groups, Bogomolov developed a lot of techniques and
interesting results. Here is one of his results.

%---------------------t1.8
\begin{theorem} \label{t1.8}
{\rm (1) \cite[Lemma 4.11]{Bo}} If $G$ is a $p$-group with
$B_0(G)\ne 0$ and $G/[G,G]\simeq C_p\times C_p$, then $p\ge 5$ and
$|G|>p^7$.

{\rm (2) \cite[Lemma 5.6; BMP, Corollary 2.11]{Bo}} If $G$ is a
$p$-group of order $\le p^5$, then $B_0(G)=0$.
\end{theorem}

Because of part (2) of the above theorem, Bogomolov proposed to
classify all the groups $G$ with $|G|=p^6$ satisfying $B_0(G)\ne 0$
\cite[page 479]{Bo}.

It came as a surprise that Moravec's recent paper \cite{Mo}
disproved the above Theorem \ref{t1.8}.

%---------------------------t1.9
\begin{theorem}[{Moravec \cite[Section 5]{Mo}}] \label{t1.9}
If $G$ is a group of order $243$, then $B_0(G)\ne 0$ if and only if
$G=G(243,i)$ with $28\le i\le 30$, where $G(243,i)$ is the $i$-th
group among groups of order $243$ in the database of GAP.
\end{theorem}

Moravec proves Theorem \ref{t1.9} by using computer calculations.
No theoretic proof is given. A file of the GAP functions and
commands for computing $B_0(G)$ can be found at Moravec's website
\verb"www.fmf.uni-1j.si/~moravec/b0g.g". Recently, using this
computer package, Moravec was able to classify all groups $G$ of
order $5^5$ and $7^5$ such that $B_0(G)\ne 0$.

Before stating the main result of this paper, we recall the
classification of $p$-groups of order $\le p^6$ and introduce the
notion of isoclinism.

A list of groups of order $2^5$ (resp.\ $3^5$, $5^5$, $7^5$) can be
found in the database of GAP. However the classification of groups
of order $p^5$ dated back to Bagnera (1898), Bender (1927), R. James
(1980), etc.\ \cite{Ba,Be,Ja}, although some minor errors might
occur in the classification results finished before the
computer-aided time. For example, in Bender's classification of
groups of order $3^5$, one group is missing, i.e.\ the group
$\Delta_{10} (2111) a_2$ which was pointed by \cite[page 613]{Ja}. A
beautiful formula for the total number of the groups of order $p^5$,
for $p\ge 3$, was found by Bagnera \cite{Ba} as
\[
2p+61+\gcd \{4,p-1\}+2\gcd\{3,p-1\}.
\]

Note that the above formula is correct only when $p \ge 5$ (see
the second paragraph of Section 4).

On the other hand, groups of order $2^n$ ($n\le 6$) were
classified by M.\ Hall and Senior \cite{HaS}. There are 267 groups
of order $2^6$ in total. Groups of order $2^7$ were classified by
R.\ James, Newman and O'Brien \cite{JNOB}.

%----------------------------------d1.10
\begin{defn} \label{d1.10}
Two $p$-groups $G_1$ and $G_2$ are called isoclinic if there exist
group isomorphisms $\theta\colon G_1/Z(G_1) \to G_2/Z(G_2)$ and
$\phi\colon [G_1,G_1]\to [G_2,G_2]$ such that $\phi([g,h])$
$=[g',h']$ for any $g,h\in G_1$ with $g'\in \theta(gZ(G_1))$, $h'\in
\theta(hZ(G_1))$ (note that $Z(G)$ and $[G,G]$ denote the center and
the commutator subgroup of the group $G$ respectively).

For a prime number $p$ and a fixed integer $n$, let $G_n(p)$ be the
set of all non-isomorphic groups of order $p^n$. In $G_n(p)$
consider an equivalence relation: two groups $G_1$ and $G_2$ are
equivalent if and only if they are isoclinic. Each equivalence class
of $G_n(p)$ is called an isoclinism family.
\end{defn}

%%%%%%%%%%%%%%%%%%%%%%%%%%%%%%%%%%%%%%%%%%%%%%%%%%%%%%%%%%%%%%%%

\begin{question} \label{quest-iso}
Let $G_1$ and $G_2$ be isoclinic $p$-groups. Is it true that the
fields $k(G_1)$ and $k(G_2)$ are stably isomorphic, or, at least,
that $B_0(G_1)$ is isomorphic to $B_0(G_2)$?
\end{question}

According to a private communication from Bogomolov, one should
expect an affirmative answer even within larger classes of groups.
Our results for groups of order $p^5$ confirm these expectations.

%%%%%%%%%%%%%%%%%%%%%%%%%%%%%%%%%%%%%%%%%%%%%%%%%%%%%%%%%%%%%%%%%

%For example,
If $p$ is an odd prime number, then there are precisely 10
isoclinism families for groups of order $p^5$; each family is
denoted by $\Phi_i$, $1\le i\le 10$ \cite[pages 619--621]{Ja}. As
for groups of order 64, there are 27 isoclinism families %; they are
%simply called the $i$-th family where $1\le i\le 27$
\cite[page~147]{JNOB}.

The main result of the present paper is the following theorem.

%-----------------------------t1.11
\begin{theorem} \label{t1.11}
Let $p$ be any odd prime number, $G$ be a group of order $p^5$. Then
$B_0(G)\ne 0$ if and only if $G$ belongs to the isoclinism family
$\Phi_{10}$. Each group $G$ in the family $\Phi_{10}$ satisfies the
condition $G/[G,G] \simeq C_p\times C_p$. There are precisely $3$
groups in this family if $p=3$. For $p\ge 5$, the total number of
non-isomorphic groups in this family is
\[
1+\gcd\{4,p-1\}+\gcd \{3,p-1\}.
\]
\end{theorem}

Note that, for $p=3$, the isoclinism family $\Phi_{10}$ consists of
the groups $\Phi_{10}(2111) a_r$ (where $r=0,1$) and $\Phi_{10}(5)$
\cite[page 621]{Ja}, which are just the groups $G(3^5,i)$ with
$28\le i\le 30$ in the GAP code numbers. This confirms the
computation of Moravec \cite{Mo}. Similarly, when $p=5$, the
isoclinism family $\Phi_{10}$ consists of the groups $G(5^5,i)$ with
$33\le i\le 38$; when $p=7$, the isoclinism family consists of the
groups $G(7^5,i)$ with $37\le i\le 42$. They agree with Moravec's
computer results.

We use the computer package provided by Moravec to study groups of
order $11^5$. We find that, for a group $G$ of order $11^5$, $B_0(G)
\neq 0$ if and only if $G \simeq G(11^5, i)$ with $39 \le i \le 42$,
also confirming the above Theorem \ref{t1.11}.

It may be interesting to record the computing time to determine
$B_0(G)$ for all $p$-groups of order $p^5$ with $p=3,5,7,11$. When
$p=3,5,7$, it requires only $20$ seconds, one hour and two days
respectively. When $p=11$, it requires more than one month by
parallel computing at four cores.

As a corollary of Theorem \ref{t1.11}, we record the following
result.

%----------------------------t1.12
\begin{theorem} \label{t1.12}
Let $n$ be a positive integer and $k$ be a field with
$\gcd\{|G|,\fn{char}k\}=1$. If $2^6\mid n$ or $p^5\mid n$ for some
odd prime number $p$, then there is a group $G$ of order $n$ such
that $B_0(G)\ne 0$. In particular, $k(G)$ is not stably
$k$-rational; when $k$ is an infinite field, $k(G)$ is not even
retract $k$-rational.
\end{theorem}

See Theorem \ref{t5.7} for anther application of Theorem
\ref{t1.11}.

For completeness, we record the result for groups of order $2^6$.
Recall that there are 267 non-isomorphic groups of order $2^6$ and
27 isoclinism families in total \cite{JNOB}.

%----------------------------t1.13
\begin{theorem}[Chu, Hu, Kang and Kunyavskii \cite{CHKK}] \label{t1.13}
Let $G$ be a group of order $2^6$.

{\rm (1)} $B_0(G)\ne 0$ if and only if $G$ belongs to the $16$th
isoclinism family, i.e.\ $G=G(2^6,i)$ where $149\le i\le 151$,
$170\le i\le 172$, $177\le i\le 178$, or $i=182$.

{\rm (2)} If $B_0(G)= 0$ and $k$ is an algebraically closed field
with $\fn{char}k\ne 2$, then $k(G)$ is rational over $k$ except
possibly for groups $G$ belonging to the $13$rd isoclinism family,
i.e.\ $G=G(2^6,i)$ with $241\le i\le 245$.
\end{theorem}

Finally we mention a recent result which supplements Moravec's
result in Theorem \ref{t1.9}.

%-------------------------t1.15
\begin{theorem}[Chu, Hoshi, Hu and Kang \cite{CHHK}] \label{t1.15}
Let $G$ be a group of order $3^5$ and of exponent $e$. If $k$ is a
field containing a primitive $e$-th root of unity and $B_0(G)=0$,
then $k(G)$ is rational over $k$.
\end{theorem}

We explain briefly the idea of the proof of Theorem \ref{t1.11}.
Let $G$ be a group of order $p^5$ where $p$ is an odd prime
number. To show that $B_0(G)=0$, we apply Theorems
\ref{t3.3}--\ref{t3.6} or some ``standard" techniques. (Note that
Theorem \ref{t4.2} is interesting by its own.) For the proof of
$B_0(G)=0$ when $G$ belongs to the isoclinism family $\Phi_6$, we
use the $7$-term cohomology exact sequence in \cite{DHW}, see
Theorems \ref{t5.6} and \ref{t5.3}.

On the other hand, to show that $B_0(G)\ne 0$, we find suitable
generators and relations for $G$. It turns out that $B_0(G)\ne 0$ if
some relations are satisfied (see Lemma \ref{l2.2}). All the groups
in the isoclinism family $\Phi_{10}$ satisfy these relations. %On the
%other hand,
Lemma \ref{l2.2} relies on the 5-term exact sequence of Hochschild
and Serre \cite{HS}
\begin{align*}
0 &\to H^1(G/N,\bm{Q}/\bm{Z})\to H^1(G,\bm{Q}/\bm{Z})\to H^1(N,\bm{Q}/\bm{Z})^G \\
&\to H^2(G/N,\bm{Q}/\bm{Z}) \xrightarrow{\psi} H^2(G,\bm{Q}/\bm{Z})
\end{align*}
where $\psi$ is the inflation map. The crux of showing $B_0(G)\ne
0$ is to prove that the image of $\psi$ is non-zero and is
contained in $B_0(G)$.

The paper is organized as follows. In Section 2, we prove that
$B_0(G)\ne 0$ if $G$ belongs to the isoclinism family $\Phi_{10}$.
Then we give a proof of Theorem \ref{t1.12}. Section 3 contains
some rationality criteria or previous results for showing
$B_0(G)=0$. Section 4 is devoted to the proof of $B_0(G)=0$ if $G$
belongs to the isoclinism family $\Phi_i$ where $1\le i\le 9$ and
$i\ne 6$. The case of $\Phi_6$ is postponed till Section 5.

\begin{idef}{Standing notations.}
Throughout this paper, $k$ is a field, $\zeta_n$ denotes a
primitive $n$-th root of unity. Whenever we write $\zeta_n\in k$
(resp. $\gcd\{n,\fn{char}k\}=1$), it is understood that either
$\fn{char}k=0$ or $\fn{char}k=l>0$ with $l\nmid n$. When $k$ is an
algebraically closed field, $\mu$ denotes the set of all roots of
unity, i.e.\ $\mu=\{\alpha\in k\backslash \{0\}: \alpha^n=1$ for
some integer $n$ depending on $\alpha\}$. If $G$ is a group,
$Z(G)$ and $[G,G]$ denote the center and the commutator subgroup
of $G$ respectively. If $g,h\in G$, we define
$[g,h]=g^{-1}h^{-1}gh\in G$. When $N$ is a normal subgroup of $G$
and $g\in G$, the element $\bar{g}\in G/N$ denotes the image of
$g$ in the quotient group $G/N$. The exponent of $G$ is defined as
$\fn{lcm}\{\fn{ord}(g):g\in G\}$ where $\fn{ord}(g)$ is the order
of the element $g$. We denote by $C_n$ the cyclic group of order
$n$. A group $G$ is called a bicyclic group if it is either a
cyclic group or a direct product of two cyclic groups. When we
write cohomology groups $H^q(G,\mu)$ or $H^q(G,\bm{Q}/\bm{Z})$, it
is understood that $\mu$ and $\bm{Q}/\bm{Z}$ are trivial
$G$-modules.

For emphasis, recall that the field $k(G)$ was defined in the first
paragraph of this section. The group $G(n,i)$ is the $i$-th group
among the groups of order $n$ in GAP. The version of GAP we refer to
in this paper is GAP4, Version: 4.4.12 \cite{GAP}. All the groups
$G$ in this paper are finite.
\end{idef}

%--------------------------------------------------------S2
\section{\boldmath Groups in the isoclinism family $\Phi_{10}$}

%Throughout this paper, when $G$ is a group, $g,h\in G$,
%we will denote by $[g,h]$ the element $g^{-1}h^{-1}gh$.

We start with a general lemma.

%-------------------------l2.1
\begin{lemma} \label{l2.1}
Let $G$ be a finite group, $N$ be a normal subgroup of $G$. Assume
that {\rm (i)} $\fn{tr}\colon H^1(N,\bm{Q}/\bm{Z})^G\to
H^2(G/N,\bm{Q}/\bm{Z})$ is not surjective where $\fn{tr}$ is the
transgression map, and {\rm (ii)} for any bicyclic subgroup $A$ of
$G$, the group $AN/N$ is a cyclic subgroup of $G/N$. Then $B_0(G)\ne
0$.
\end{lemma}

\begin{proof}
Consider the Hochschild--Serre 5-term exact sequence
\begin{align*}
0 &\to H^1(G/N,\bm{Q}/\bm{Z})\to H^1(G,\bm{Q}/\bm{Z})\to H^1(N,\bm{Q}/\bm{Z})^G \\
&\xrightarrow{\fn{tr}} H^2(G/N,\bm{Q}/\bm{Z}) \xrightarrow{\psi} H^2(G,\bm{Q}/\bm{Z})
\end{align*}
where $\psi$ is the inflation map \cite{HS}.

Since tr is not surjective, we find that $\psi$ is not the zero map.
Thus $\fn{Image}(\psi)\ne 0$.

We will show that $\fn{Image}(\psi)\subset B_0(G)$. By definition,
it suffices to show that, for any bicyclic subgroup $A$ of $G$, the
composite map $H^2(G/N,\bm{Q}/\bm{Z})\xrightarrow{\psi}
H^2(G,\bm{Q}/\bm{Z})\xrightarrow{\fn{res}} H^2(A,\bm{Q}/\bm{Z})$
becomes the zero map where res is the restriction map. Consider the
following commutative diagram
\[
\begin{array}{c}
H^2(G/N,\bm{Q}/\bm{Z}) \xrightarrow{\psi} H^2(G,\bm{Q}/\bm{Z}) \xrightarrow{\fn{res}} H^2 (A,\bm{Q}/\bm{Z}) \\[4pt]
{}^{\psi_0} \bigg\downarrow \hspace*{5.5cm} \bigg\uparrow {}^{\psi_1} \\[-2pt]
H^2(AN/N,\bm{Q}/\bm{Z}) \stackrel{\widetilde{\psi}}{\simeq} H^2(A/A\cap N,\bm{Q}/\bm{Z})
\end{array}
\]
where $\psi_0$ is the restriction map, $\psi_1$ is the inflation
map, $\widetilde{\psi}$ is the natural isomorphism.

Since $AN/N$ is cyclic, write $AN/N\simeq C_m$ for some integer $m$.
It is well-known that $H^2(C_m,\bm{Q}/\bm{Z})=0$ (see, e.g.,
\cite[page~37, Corollary~2.2.12]{Kar}). Hence $\psi_0$ is the zero
map. Thus $\text{res} \circ \psi\colon H^2(G/N,\bm{Q}/\bm{Z})\to
H^2(A,\bm{Q}/\bm{Z})$ is also the zero map.

As $\fn{Image}(\psi)\subset B_0(G)$ and $\fn{Image}(\psi)\ne 0$, we
find that $B_0(G)\ne 0$.
\end{proof}

%-------------------------------l2.2
\begin{lemma} \label{l2.2}
Let $p\ge 3$ and $G$ be a $p$-group of order $p^5$ generated by
$f_i$ where $1\le i\le 5$. Suppose that, besides other relations,
the generators $f_i$ satisfy the following conditions:
\begin{enumerate}
\item[{\rm (i)}]
$f_4^p=f_5^p=1$, $f_5\in Z(G)$,
\item[{\rm (ii)}]
$[f_2,f_1]=f_3$, $[f_3,f_1]=f_4$, $[f_4,f_1]=[f_3,f_2]=f_5$, $[f_4,f_2]=[f_4,f_3]=1$, and
\item[{\rm (iii)}]
$\langle f_4,f_5\rangle \simeq C_p\times C_p$, $G/\langle f_4,f_5\rangle$ is a non-abelian group of order $p^3$ and of exponent $p$.
\end{enumerate}
Then $B_0(G)\ne 0$.
\end{lemma}

%------------------------------r2.3
\begin{remark}
When $p=2$ and $G/N$ is a non-abelian group of order 8, then
$H^2(G/N$, $\bm{Q}/\bm{Z})=0$ or $C_2$ \cite[page~138, Theorem~
3.3.6]{Kar}. Thus $\fn{tr}\colon  H^1(N,\bm{Q}/\bm{Z})^G \to
H^2(G/N$, $\bm{Q}/\bm{Z})$ in Lemma \ref{l2.2} may become
surjective. This is the reason why we assume $p\ge 3$ in this lemma.
\end{remark}

\begin{proof}
Choose $N=\langle f_4,f_5\rangle$.
We will check the conditions in Lemma \ref{l2.1} are satisfied. Thus $B_0(G)\ne 0$.

\bigskip
Step 1. Since $N\simeq C_p\times C_p$, we find that
$H^1(N,\bm{Q}/\bm{Z})\simeq C_p\times C_p$.

Define $\varphi_1,\varphi_2\in H^1(N,\bm{Q}/\bm{Z})=
\fn{Hom}(N,\bm{Q}/\bm{Z})$ by $\varphi_1(f_4)=1/p$,
$\varphi_1(f_5)=0$, $\varphi_2(f_4)=0$, $\varphi_2(f_5)=1/p$.
Clearly $H^1(N,\bm{Q}/\bm{Z})=\langle \varphi_1,\varphi_2\rangle$.

The action of $G$ on $\varphi_1$, $\varphi_2$ are given by
$\side{f_1}{\varphi_1}(f_4)=\varphi_1(f_1^{-1}f_4f_1)=\varphi_1(f_4f_5)=\varphi_1(f_4)\break
+\varphi_1(f_5)=1/p$,
$\side{f_1}{\varphi_1}(f_5)=\varphi_1(f_1^{-1}f_5f_1)=\varphi_1(f_5)=0$.
Thus $\side{f_1}{\varphi_1}=\varphi_1$. Similarly,
$\side{f_1}{\varphi_2} (f_4)=1/p$, $\side{f_1}{\varphi_2} (f_5)=1/p$
and $\side{f_1}{\varphi_2}=\varphi_1+\varphi_2$.

For any $\varphi\in
H^1(N,\bm{Q}/\bm{Z})=\langle\varphi_1,\varphi_2\rangle \simeq
C_p\times C_p$, write $\varphi=a_1\varphi_1+a_2\varphi_2$ for some
integers $a_1,a_2\in \bm{Z}$ (modulo $p$). Since
$\side{f_1}{\varphi}=\side{f_1}(a_1\varphi_1+a_2\varphi_2)=a_1(\side{f_1}{\varphi_1})+a_2(\side{f_1}{\varphi_2})=(a_1+a_2)\varphi_1+a_2\varphi_2$,
we find that $\side{f_1}{\varphi}=\varphi$ if and only if $a_2=0$,
i.e.\ $\varphi\in \langle \varphi_1\rangle$. On the other hand, it
is easy to see that
$\side{f_2}{\varphi_1}=\varphi_1=\side{f_3}{\varphi_1}$ and
therefore $\varphi_1\in H^1(N,\bm{Q}/\bm{Z})^G$. We find
$H^1(N,\bm{Q}/\bm{Z})^G=\langle\varphi_1\rangle \simeq C_p$.

By \cite[Proposition 6.3; Kar, page 138, Theorem 3.3.6]{Le}, since
$G/N$ is a non-abelian group of order $p^3$ and of exponent $p$, we
find $H^2(G/N,\bm{Q}/\bm{Z}) \simeq C_p\times C_p$. Thus
$\fn{tr}\colon H^1(N,\bm{Q}/\bm{Z})^G \to H^2(G/N,\bm{Q}/\bm{Z})$ is
not surjective. Hence the first condition of Lemma \ref{l2.1} is
verified.

\bigskip
Step 2.
We will verify the second condition of Lemma \ref{l2.1},
i.e.\ for any bicyclic subgroup $A$ of $G$, $AN/N$ is a cyclic group.

Before the proof, we list the following formulae which are consequences of the commutator relations,
i.e.\ relations (ii) of this lemma.
The proof of these formulae is routine and is omitted.

For $1\le i,j\le p-1$, $f_4^if_1^j=f_1^jf_4^if_5^{ij}$,
$f_3^if_2^j=f_2^jf_3^if_5^{ij}$, and
\[
f_3^if_1^j=f_1^jf_3^if_4^{ij}f_5^{i\cdot\binom{j}{2}},\quad
f_2^if_1^j=f_1^jf_2^if_3^{ij}f_4^{i\cdot\binom{j}{2}}f_5^{i\cdot\binom{j}{3}+\binom{i}{2}\cdot j}
\]
where $\binom{a}{b}$ denotes the binomial coefficient when $a\ge b\ge 1$ and we adopt the convention $\binom{a}{b}=0$ if $1\le a<b$.

Moreover, in $G/N$, $(\bar{f}_1^j \bar{f}_2^i)^e=\bar{f}_1^{ej} \bar{f}_2^{ei} \bar{f}_3^{\binom{e}{2}\cdot ij}$ for $1\le i,j\le p-1$, $1\le e\le p$.

\bigskip
Step 3.
Let $A=\langle h_1,h_2 \rangle$ be a bicyclic subgroup of $G$.
We will show that $AN/N$ is cyclic in $G/N$.

Since $AN/N$ is abelian and $G/N$ is not abelian, we find that
$AN/N$ is a proper subgroup of $G/N$ which is of order $p^3$.

If $|AN/N|\le p$, then $AN/N$ is cyclic.
From now on, we will assume $AN/N$ is an order $p^2$ subgroup and try to find a contradiction.

In $G/N$, write $\bar{h}_1=\bar{f}_1^{a_1} \bar{f}_2^{a_2}
\bar{f}_3^{a_3}$, $\bar{h}_2=\bar{f}_1^{b_1} \bar{f}_2^{b_2}
\bar{f}_3^{b_3}$ for some integers $a_j$, $b_j$ (recall that
$G/N=\langle \bar{f}_1,\bar{f}_2,\bar{f}_3\rangle$ and $A=\langle
h_1,h_2 \rangle$). After suitably changing the generators $h_1$ and
$h_2$, we will show that there are only three possibilities:
$(\bar{h}_1,\bar{h}_2)=(\bar{f}_2,\bar{f}_3)$,
$(\bar{f}_1\bar{f}_3^{a_3},\bar{f}_2\bar{f}_3^{b_3})$,
$(\bar{f}_1\bar{f}_2^{a_2},\bar{f}_3)$ for some integers $a_2$,
$a_3$, $b_3$.

Suppose $\bar{h}_1=\bar{f}_1^{a_1} \bar{f}_2^{a_2} \bar{f}_3^{a_3}$ and $\bar{h}_2=\bar{f}_1^{b_1} \bar{f}_2^{b_2} \bar{f}_3^{b_3}$ as above.
If $a_1=b_1=0$, then $\langle \bar{h}_1,\bar{h}_2\rangle\break =\langle \bar{f}_2,\bar{f}_3\rangle$.
Thus after changing the generating elements $h_1$, $h_2$,
we may assume that $\bar{h}_1=\bar{f}_2$, $\bar{h}_2=\bar{f}_3$.
This is the first possibility.

If $a_1 \not\equiv 0$ or $b_1\not\equiv 0$ (mod $p$),
we may assume $1\le a_1\le p-1$.
Find an integer $e$ such that $1\le e\le p-1$ and $a_1 e\equiv 1$ (mod $p$).
Use the formulae in Step 2,
we get $\bar{h}_1^e=\bar{f}_1 \bar{f}_2^{c_2} \bar{f}_3^{c_3}$.
Since $\langle h_1,h_2\rangle=\langle h_1^e,h_2\rangle$,
without loss of generality,
we may assume that $\bar{h}_1=\bar{f}_1\bar{f}_2^{a_2} \bar{f}_3^{a_3}$ (i.e.\ $a_1=1$ from the beginning).

Since $\langle h_1,h_2\rangle =\langle h_1,(h_1^{b_1})^{-1} h_2\rangle$,
we may assume $\bar{h}_1=\bar{f}_1 \bar{f}_2^{a_2} \bar{f}_3^{a_3}$ and $\bar{h}_2=\bar{f}_2^{b_2} \bar{f}_3^{b_3}$.

In the case $1\le b_2\le p-1$, take an integer $e'$ with $1\le e'
\le p-1$ and $b_2e'\equiv 1$ (mod $p$). Use the generating set
$\langle h_1,h_2^{e'}\rangle$ for $A$. Thus we may assume
$\bar{h}_1=\bar{f}_1\bar{f}_3^{a_3}$,
$\bar{h}_2=\bar{f}_2\bar{f}_3^{b_3}$. This is the second
possibility.

If $b_2\equiv 0$ (mod $p$), then $\bar{h}_1=\bar{f}_1\bar{f}_2^{a_2} \bar{f}_3^{a_3}$, $\bar{h}_2=\bar{f}_3^{b_3}$.
If $b_3=0$, then $AN/N$ is cyclic.
Thus $b_3\not\equiv 0$ (mod $p$).
Changing the generators again, we may assume $\bar{h}_1=\bar{f}_1\bar{f}_2^{a_2}$, $\bar{h}_2=\bar{f}_3$.
This is the third possibility.

\bigskip
Step 4.
We will show that all three possibilities in Step 3 lead to contradiction.

Suppose $\bar{h}_1=\bar{f}_2$, $\bar{h}_2=\bar{f}_3$. Write
$h_1=f_2f_4^{a_4}f_5^{a_5}$, $h_2=f_3f_4^{b_4}f_5^{b_5}$. Since
$h_1h_2=h_2h_1$, we get
$f_2f_4^{a_4}f_3f_4^{b_4}=f_3f_4^{b_4}f_2f_4^{a_4}$ (because $f_5\in
Z(G)$). Rewrite this equality with the help of the formulae in Step
2. We get $f_2f_3f_4^{a_4+b_4}=f_2f_3f_4^{a_4+b_4}f_5$, which is a
contradiction.

Suppose $\bar{h}_1=\bar{f}_1\bar{f}_3^{a_3}$,
$\bar{h}_2=\bar{f}_2\bar{f}_3^{b_3}$. In $G/N$, we have $\bar{h}_1
\bar{h}_2=\bar{h}_2\bar{h}_1$,  but it is obvious the two elements
$\bar{f}_1\bar{f}_3^{a_3}$, $\bar{f}_2\bar{f}_3^{b_3}$ do not
commute. Done.

Suppose $\bar{h}_1=\bar{f}_1\bar{f}_2^{a_2}$, $\bar{h}_2=\bar{f}_3$.
Write $h_1=f_1f_2^{a_2}f_4^{a_4}f_5^{a_5}$, $h_2=f_3f_4^{b_4}f_5^{b_5}$.
Use the fact $h_1h_2=h_2h_1$.
It is easy to find a contradiction.
\end{proof}

%-------------------------------t2.3
\begin{theorem} \label{t2.3}
Let $p$ be an odd prime number and $G$ be a group of order $p^5$ belonging to the isoclinism family $\Phi_{10}$.
Then $B_0(G)\ne 0$.
\end{theorem}

\begin{proof}
Apply Lemma \ref{l2.2}. It suffices to show that $G$ satisfies
conditions (i), (ii), (iii) in Lemma \ref{l2.2}.

\begin{Case}{1} $p=3$. \end{Case}

It is routine to verify that the groups $\Phi_{10}(1^5)$,
$\Phi_{10}(2111)a_0$, $\Phi_{10}(2111)a_1$ in \cite[page~621]{Ja}
are isomorphic to $G(3^5,28)$, $G(3^5,29)$, $G(3^5,30)$
respectively. All these three groups $G(3^5,i)$ with $28\le i\le 30$
can be defined as
\begin{gather*}
G(3^5,i)=\langle f_1,f_2,f_3,f_4,f_5\rangle, \quad Z(G(3^5,i))=\langle f_5\rangle, \\
[f_2,f_1]=f_3,\ [f_3,f_1]=f_4,\ [f_4,f_1]=[f_3,f_2]=f_5,\ [f_4,f_2]=[f_4,f_3]=1
\end{gather*}
with additional relations
\begin{alignat*}{2}
f_1^3 &= f_4^3=f_5^3=1,\ f_2^3=f_4^{-1},\ f_3^3=f_5^{-1} & \text{ for } & G(3^5,28), \\
f_4^3 &= f_5^3=1,\ f_1^3=f_5,\ f_2^3=f_4^{-1},\ f_3^3=f_5^{-1} & \text{ for } & G(3^5,29), \\
f_4^3 &= f_5^3=1,\ f_1^3=f_5^{-1},\ f_2^3=f_4^{-1},\ f_3^3=f_5^{-1} & \text{ for } & G(3^5,30).
\end{alignat*}

\begin{Case}{2} $p\ge 5$. \end{Case}

The group $G=\Phi_{10}(1^5)$ in \cite[page~621]{Ja} is defined as
\begin{gather*}
G=\langle f_1,f_2,f_3,f_4,f_5\rangle, \quad Z(G)=\langle f_5\rangle, \\
f_i^p=1 \text{ for }1\le i\le 5, \\
[f_2,f_1]=f_3,\ [f_3,f_1]=f_4,\ [f_4,f_1]=[f_3,f_2]=f_5,\ [f_4,f_2]=[f_4,f_3]=1.
\end{gather*}
The group $G=\Phi_{10}(2111)a_r$ in \cite[page~621]{Ja} is defined
as
\begin{gather*}
G=\langle f_1,f_2,f_3,f_4,f_5\rangle, \quad Z(G)=\langle f_5\rangle, \\
f_1^p=f_5^{\alpha^r},\ f_i^p=1 \text{ for }2\le i\le 5, \\
[f_2,f_1]=f_3,\ [f_3,f_1]=f_4,\ [f_4,f_1]=[f_3,f_2]=f_5,\ [f_4,f_2]=[f_4,f_3]=1
\end{gather*}
where $\alpha$ is the smallest positive integer which is a primitive root (mod $p$) and $0\le r\le \gcd\{3,p-1\}-1$.

The group $G=\Phi_{10} (2111) b_r$ in \cite[page~621]{Ja} is defined
as
\begin{gather*}
G=\langle f_1,f_2,f_3,f_4,f_5\rangle,\quad Z(G)=\langle f_5\rangle, \\
f_2^p=f_5^{\alpha^r},\ f_1^p=f_i^p=1 \text{ for }3\le i\le 5, \\
[f_2,f_1]=f_3,\ [f_3,f_1]=f_4,\ [f_4,f_1]=[f_3,f_2]=f_5,\ [f_4,f_2]=[f_4,f_3]=1
\end{gather*}
where $\alpha$ is the smallest positive integer which is a primitive root (mod $p$) and $0\le r\le \gcd\{3,p-1\}-1$.
\end{proof}

%-------------------------r
\begin{remark}
In the proof of \cite[Lemma 5.6, page~478]{Bo}, Bogomolov tried to
prove that there do not exist $p$-groups $G$ of order $p^5$ with
$B_0(G)\ne 0$. He assumed that the commutator group $[G,G]$ was
abelian and discussed three situations when the order of $G/[G,G]$
was $p^2$, $p^3$, or $\ge p^4$ (in general, if $G$ is a non-abelian
group of order $p^5$, then $[G,G]$ is abelian, since $G$ has an
abelian normal subgroup of order $p^3$ by a theorem of Burnside).
The case when $G/[G,G]=p^2$ was reduced to \cite[Lemma 4.11, page~
478]{Bo} (see the first part of Theorem \ref{t2.3}). But this lemma
is disproved in the proof of the above theorem.
\end{remark}

\bigskip
\begin{proof}[Proof of Theorem \ref{t1.12}]
Suppose that $p^5\mid n$ for some odd prime number $p$. Write
$n=p^5m$. By Theorem \ref{t2.3} choose a group $G_0$ of order
$p^5$ satisfying $B_0(G_0)\ne 0$. Define $G=G_0\times C_m$.

We will prove that $k(G)$ is not stably $k$-rational (resp.\ not
retract $k$-rational if $k$ is infinite). Suppose not. Assume that
$k(G)$ is stably $k$-rational (resp.\ retract $k$-rational if $k$
is infinite). Then so is $\bar{k}(G)$ over $\bar{k}$ where
$\bar{k}$ is the algebraic closure of $k$. In particular,
$\bar{k}(G)$ is retract $\bar{k}$-rational. Since $G=G_0\times
C_m$, by \cite[Theorem 1.5; Ka4, Lemma 3.4]{Sa1}, we find that
$\bar{k}(G_0)$ is retract $\bar{k}$-rational. This implies
$B_0(G)=0$ by Lemma \ref{l1.3}. A contradiction.

In case $2^6\mid n$, the proof is similar by applying Theorem \ref{t1.5}.
\end{proof}

%-------------------------------------------S3
\section{Some reduction theorems}

We recall several known results in this section.

%-----------------------t3.1
\begin{theorem}[{Ahmad, Hajja and Kang \cite[Theorem 3.1]{AHK}}] \label{t3.1}
Let $L$ be any field, $L(x)$ the rational function field in one variable over $L$,
and $G$ a finite group acting on $L(x)$.
Suppose that, for any $\sigma\in G$,
$\sigma(L)\subset L$ and $\sigma(x)=a_\sigma \cdot x+b_\sigma$ where $a_\sigma, b_\sigma \in L$ and $a_\sigma \ne 0$.
Then $L(x)^G=L^G(f)$ for some polynomial $f\in L[x]$.
In fact, if $m=\min \{\fn{deg} g(x):g(x)\in L[x]^G\backslash L^G\}$,
any polynomial $f\in L[x]^G$ with $\fn{deg} f=m$ satisfies the property $L(x)^G=L^G(f)$.
\end{theorem}

%-----------------------t3.2
\begin{theorem}[{Hajja and Kang \cite[Theorem 1]{HK}}] \label{t3.2}
Let $G$ be a finite group acting on $L(x_1,\ldots,x_n)$,
the rational function field in $n$ variables over a field $L$.
Suppose that
\begin{enumerate}
\item[{\rm (i)}]
for any $\sigma \in G$, $\sigma(L)\subset L$,
\item[{\rm (ii)}]
the restriction of the action of $G$ to $L$ is faithful,
\item[{\rm (iii)}]
for any $\sigma\in G$,
\[
\begin{pmatrix} \sigma(x_1) \\ \sigma(x_2) \\ \vdots \\ \sigma(x_n) \end{pmatrix}
=A(\sigma) \cdot \begin{pmatrix} x_1 \\ x_2 \\ \vdots \\ x_n \end{pmatrix}+B(\sigma)
\]
where $A(\sigma)\in GL_n(L)$ and $B(\sigma)$ is an $n\times 1$ matrix over $L$.
\end{enumerate}

Then there exist elements $z_1,\ldots,z_n\in L(x_1,\ldots,x_n)$ so that $L(x_1,\ldots,x_n)=L(z_1$, $\ldots,z_n)$ and $\sigma(z_i)=z_i$ for any $\sigma\in G$, any $1\le i\le n$.
\end{theorem}

%---------------------------t3.3
\begin{theorem}[{Fischer \cite[Theorem 6.1]{Sw}}] \label{t3.3}
Let $G$ be a finite abelian group of exponent $e$, and let $k$ be
a field containing a primitive $e$-th root of unity. Then $k(G)$
is rational over $k$.
\end{theorem}

%-----------------------t3.4
\begin{theorem}[{Kang and Plans \cite[Theorem 1.3]{KP}}] \label{t3.4}
Let $k$ be any field, $G_1$ and $G_2$ be two finite groups.
If $k(G_1)$ and $k(G_2)$ are rational over $k$, then so is $k(G_1\times G_2)$ over $k$.
\end{theorem}

%------------------------t3.5
\begin{theorem} \label{t3.5}
Let $k$ be a field and $G$ be a finite group.
Assume that {\rm (i)} $G$ contains an abelian normal subgroup $H$ such that $G/H$ is a cyclic group,
and {\rm (ii)} $k$ contains a primitive $e$-th root of unity where $e=\exp (G)$.

{\rm (1) (Bogomolov \cite[Lemma 4.9]{Bo})} If $k$ is algebraically
closed, then $B_0(G)=0$.

{\rm (2) (Kang \cite[Theorem 5.10]{Ka})}
If $k$ is an infinite field, then $k(G)$ is retract $k$-rational.
In particular, $B_0(G)=0$.

{\rm (3) (Kang \cite[Theorem 2.2]{Ka1})}
If $\bm{Z}[\zeta_n]$ is a unique factorization domain where $n=|G/H|$,
then $k(G)$ is rational over $k$.
\end{theorem}

%------------------------t3.6
\begin{theorem}[{Kang \cite[Theorem 1.8]{Ka2}}] \label{t3.6}
Let $n\ge 3$ and $G$ be a non-abelian group of order $p^n$ such that $G$ has a cyclic subgroup of index $p^2$.
If $k$ is a field containing a primitive $p^{n-2}$-th root of unity,
then $k(G)$ is rational over $k$.
\end{theorem}

%-------------------------t3.7
\begin{theorem} \label{t3.7}
Let $L$ be any field containing a field $k$,
$L(x)$ be the rational function field of one variable over $L$.

{\rm (1) (Saltman \cite[Proposition 3.6; Ka4, Lemma 3.4]{Sa2})} If
$k$ is an infinite field, then $L$ is retract $k$-rational if and
only if so is $L(x)$ over $k$.

{\rm (2) (Saltman \cite[Section 2; Ka4, Theorem 3.2]{Sa5})} The
natural map $\fn{Br}_{v,k}(L)\to \fn{Br}_{v,k}(L(x))$ is an
isomorphism.
\end{theorem}

The following is an elementary result in group theory, whose proof is omitted.

%---------------------------l3.8
\begin{lemma} \label{l3.8}
Let $G$ be a finite $p$-group.
If $H$ is a normal subgroup of $G$ and $H\ne \{1\}$, then $H\cap Z(G)\ne \{1\}$.
\end{lemma}

%---------------------------l3.9
\begin{lemma} \label{l3.9}
Let $G$ be a finite $p$-group, $Z(G)$ be its center. Let $\theta:
G\to GL(W)$ be a linear representation of $G$ where $W$ is a
finite-dimensional vector space over some field $k$. Assume that,
for any $g\in Z(G) \backslash \{1\}$, $\theta(g)\ne 1$. Then
$\theta$ is a faithful representation of $G$, i.e.\ $\theta$ is
injective.
\end{lemma}

\begin{proof}
Let $N=\fn{Ker}(\theta)$.
If $N\ne \{1\}$, then $N\cap Z(G)\ne \{1\}$ by Lemma \ref{l3.8}.
It follows that there is some $g\in Z(G)\backslash \{1\}$ with $\theta(g)=1$.
A contradiction.
\end{proof}

We recall the definitions of $G$-lattices and purely monomial actions.

%-------------------------d3.10
\begin{defn} \label{d3.10}
Let $G$ be a finite group.
A $G$-lattice $M$ is a finitely generated $\bm{Z}[G]$-module which is $\bm{Z}$-free as an abelian group,
i.e.\ $M=\bigoplus_{1\le i\le n} \bm{Z}\cdot x_i$ with a $\bm{Z}[G]$-module structure.

If $k$ is a field and $M=\bigoplus_{1\le i\le n} \bm{Z}\cdot x_i$ is a $G$-lattice,
define $k(M)=k(x_1,\ldots,x_n)$ the rational function field over $k$ with $G$ acting by $k$-automorphisms defined as follows:
For any $\sigma\in G$, if $\sigma\cdot x_j=\sum_{1\le i\le n} a_{ij}x_i$ in $M$,
then $\sigma\cdot x_j=\prod_{1\le i\le n} x_i^{a_{ij}}$ in $k(M)$.
The action of $G$ on $k(M)$ is called a purely monomial $k$-action \cite[Definition 1.15]{HKK}.
The fixed field of $k(M)$ under the $G$-action is denoted by $k(M)^G$.
\end{defn}

%--------------------------t3.11
\begin{theorem}[Barge \cite{Bar}] \label{t3.11}
Let $G$ be a finite group, $k$ be an algebraically closed field with $\gcd\{|G|,\fn{char}k\}=1$.
The following two statements are equivalent,

{\rm (i)} all the Sylow subgroups of $G$ are bicyclic,

{\rm (ii)} $\fn{Br}_{v,k} (k(M)^G)=0$ for all $G$-lattices $M$.
\end{theorem}

\begin{proof}
In \cite{Bar}, the above theorem is proved for the case $k=\bm{C}$
but the arguments there work in the general case.

Here is an alternative proof for the direction ``(i) $\Rightarrow$
(ii)" of the above theorem: apply \cite[Theorem 12]{Sa4}.
\end{proof}

%-------------------------------------------------S4
\section{\boldmath $B_0(G)=0$ for the groups not belonging to $\Phi_{6}$ and $\Phi_{10}$}

Let $p$ be an odd prime number and $G$ be a group of order $p^5$
belonging to the isoclinism family $\Phi_i$ where $1\le i\le 9$. We
will show that $B_0(G)=0$ in this section and the next section.

We adopt the classification of groups of order $p^5$ by R.\ James
\cite{Ja}. For groups of order $p^5$, there are in total 10
isoclinism families $\Phi_i$ where $1\le i\le 10$ \cite[pages~
619--621]{Ja}. When $p\ge 5$, the numbers of groups in the family
$\Phi_i$ where $1\le i\le 10$ are
\[
7,\ 15,\ 13,\ p+8,\ 2,\ p+7,\ 5,\ 1,\ \gcd\{3,p-1\}+2,\
\gcd\{4,p-1\}+\gcd\{3,p-1\}+1
\]
respectively. The same numbers hold true for groups of order $3^5$
except for $\Phi_6$ and $\Phi_{10}$. The numbers of groups of
order $3^5$ in $\Phi_6$ and $\Phi_{10}$ are 7 and 3 respectively.

We call the attention of the reader to two conventions of James's
paper \cite{Ja}. First the notation $\alpha_{i+1}^{(p)}$ is not
$\alpha_{i+1}^p$ in general; it is defined as
$\alpha_{i+1}^{(p)}=\alpha_{i+1}^p \alpha_{i+2}^{\binom{p}{2}}
\cdots \alpha_{i+k}^{\binom{p}{k}} \cdots \alpha_{i+p}$ where
$\alpha_{i+2},\ldots,\alpha_{i+p}$ are suitably defined
\cite[p.~614, lines 8--10]{Ja}. In particular, for the groups of
order $p^5$ with $p\ge 5$ defined in \cite[pages~619--621]{Ja},
$\alpha_{i+1}^{(p)}=\alpha_{i+1}^p$. On the other hand, when $p=3$,
the relations
$\alpha_1^{(3)}=\alpha_2^{(3)}=\alpha_3^{(3)}=\alpha_4^{(3)}=1$ for
the group $\Phi_9(2111)a$ in \cite[page 621]{Ja} are equivalent to
the relations $\alpha_1^3=\alpha_3^{-1}\alpha_4$,
$\alpha_2^3=\alpha_4^{-1}$ and $\alpha_3^3=\alpha_4^3=1$. The second
convention of \cite{Ja} is that all relations of the form
$[\alpha,\beta]=1$ are omitted from the list \cite[p.~614, lines
11--12]{Ja}.

%-----------------------------------------t4.1
\begin{theorem} \label{t4.1}
Let $p$ be an odd prime number and $G$ be a group of order $p^5$
and of exponent $e$. If $k$ is an infinite field containing a
primitive $e$-th root of unity and $G$ belongs to the isoclinism
family $\Phi_i$ where $1\le i\le 4$ or $8\le i\le 9$, then $k(G)$
is retract rational over $k$. In particular, $B_0(G)=0$.
\end{theorem}

\begin{proof}
If $G$ belongs to the isoclinism family $\Phi_i$ where $1\le i\le 4$
or $8\le i\le 9$, it is not difficult (from the list of
\cite[pages~619--621]{Ja}) to find an abelian normal subgroup $H$
such that $G/H$ is cyclic. Thus $k(G)$ is retract $k$-rational and
$B_0(G)=0$ by Theorem \ref{t3.5}. But we can say more about $k(G)$.

\bigskip
Step 1.
The groups in $\Phi_1$ are abelian groups.
If $G\in \Phi_1$, then $k(G)$ is $k$-rational by Theorem \ref{t3.3}.

\bigskip
Step 2. Some groups in $\Phi_2$ are direct products. If $G\in
\Phi_2$ and $G\simeq G_1\times G_2$ with $|G_1|,|G_2|<|G|$, then
both $k(G_1)$ and $k(G_2)$ are $k$-rational by Theorem \ref{t1.6}.
Thus $k(G)$ is $k$-rational by Theorem \ref{t3.4}.

For the other groups $G\in \Phi_2$, it is easy to verify that
$G/Z(G)\simeq C_p\times C_p$. Let $\bar{g}$ be an element of order
$p$ in $G/Z(G)$ and $g$ be a preimage of $\bar{g}$ in $G$. Then
$H=\langle Z(G),g\rangle$ is abelian and normal in $G$ with
$G/H\simeq C_p$. By Theorem \ref{t3.5}, $k(G)$ is retract
$k$-rational.

\bigskip
Step 3. If $G$ belongs to $\Phi_3$ or $\Phi_4$, it is not difficult
to show that $G$ contains an abelian normal subgroup of index $p$ by
checking the list provided in \cite[page~620]{Ja}.

Alternatively, we may use the fact asserted in Bender's paper
\cite[p.69]{Be}: If $G$ is a group of order $p^5$ (where $p\ge 3$)
with $|Z(G)|=p^2$ and $|[G,G]|\le p^2$, then $G$ contains an
abelian normal subgroup of index $p$. Assuming this fact, since
$|Z(G)|=|[G,G]|=p^2$ (if $G\in \Phi_3$ and $G$ is not a direct
product) and $|Z(G)|=|[G,G]|=p^2$ (if $G\in \Phi_4$), we are done.

In either case, apply Theorem \ref{t3.5}. We find that $k(G)$ is
retract $k$-rational.

\bigskip
Step 4. If $G\in \Phi_8$, the family $\Phi_8$ consists of only one
group $G\simeq C_{p^3} \rtimes C_{p^2}$. Apply Theorem \ref{t3.6}.
We find $k(G)$ is $k$-rational.

\bigskip
Step 5. If $G\in \Phi_9$, check the list of the generators and
relations of these groups in \cite[p.621]{Ja}. We find that these
groups $G$ are generated by elements $f_0$, $f_1$, $f_2$, $f_3$,
$f_4$ and, besides other relations, they satisfy the relations
\begin{align*}
& \begin{alignedat}{2}
[f_i,f_0]&=f_{i+1} &\text{ for }& 1\le i\le 3, \\
[f_i,f_j]&=1 & \text{ for } &1\le i,j\le 4, \text{ and}
\end{alignedat} \\
& \{f_1,f_2,f_3,f_4\} \text{ generates a subgroup of index }p.
\end{align*}

Define $H=\langle f_1,f_2,f_3,f_4\rangle$.
It follows that $H$ is an abelian normal subgroup of index $p$.
Apply Theorem \ref{t3.5}.
\end{proof}

%----------------------------t4.2
\begin{theorem} \label{t4.2}
Let $G=A\rtimes G_0$ be a finite group where $A$ and $G_0$ are
subgroups of $G$ such that {\rm (i)} $A$ is an abelian normal
subgroup of $G$ with $G_0$ acting on $A$, and {\rm (ii)} $G_0$ is
bicyclic. If $k$ is an algebraically closed field with
$\gcd\{|G|,\fn{char}k\}=1$, then $\fn{Br}_{v,k} (k(G))=0$.
\end{theorem}

\begin{proof}
Step 1. Let $V= \bigoplus_{g \in G} k \cdot x(g)$ with the
$G$-action defined by $g \cdot x(h)=x(gh)$ for any $g,h \in G$.
Then $k(G)=k(x(g): g \in G)^G$ by definition.

Consider a subspace $W= \bigoplus_{\tau \in A} k \cdot x(\tau)$.
Since $A$ is abelian, the action of $A$ on $W$ can be
diagonalized. Explicitly, there is a linear change of variables of
$W$ with $W=\bigoplus_{1\le i\le n} k\cdot x_i$ (where $n=| A |$)
such that, for all $\tau\in A$, $\tau\cdot x_i \in k\cdot x_i$ for
$1\le i\le n$. Thus we may write $\tau\cdot x_i=\chi_i(\tau)x_i$
where $\chi_i: A\to k^{\times}$ is a linear character of $A$.

For any $h \in G_0$, define $W(h)=\bigoplus_{\tau \in A} k \cdot
x(h\tau)$. Since $x(h \tau)=h \cdot x(\tau)$, we find that
$W(h)=h(W)=h(\bigoplus_{1\le i\le n} k\cdot x_i)=\bigoplus_{1 \le
i \le n} k \cdot (h \cdot x_i)$. Note that $\tau\cdot (h \cdot
x_i)= h(h^{-1}\tau h) \cdot x_i=\chi_i(h^{-1}\tau h) (h \cdot
x_i)$ for any $\tau \in A$, any $h \in G_0$.

Write $y_i(g)=g \cdot x_i$ for any $g \in G_0$, any $1 \le i \le
n$. It follows that $k(x(g): g \in G)=k(y_i(g): 1\le i\le n, g\in
G_0)$. The action of $G$ on $y_i(g)$ is given as follows: For all
$\tau\in A$, $g,h\in G_0$, $1\le i\le n$, we have
\[
\tau\cdot y_i(g)=\chi_i(g^{-1}\tau g) y_i (g),\quad h\cdot
y_i(g)=y_i(hg).
\]

It remains to show that $\fn{Br}_{v,k}(k(y_i(g):1\le i\le n, g\in
G_0)^G)=0$.

\bigskip
Step 2. Define a $G_0$-lattice $N=\bigoplus_{g\in G_0,1\le i\le n}
\bm{Z}\cdot y_i(g)$ with
\[
h\cdot y_i(g)=y_i(hg)
\]
for any $h,g\in G_0$.

Let us choose $\tau_1,\ldots,\tau_m \in A$ such that $A=\langle
\tau_1,\ldots,\tau_m\rangle$. Let $\zeta$ be a root of unity such
that $\langle \chi_i(\tau):\tau\in A,1\le i\le
n\rangle=\langle\zeta\rangle$. Regard
$\langle\zeta\rangle^m:=\langle\zeta\rangle \times \cdots \times
\langle\zeta\rangle$ (the direct product of $m$ copies of
$\langle\zeta\rangle$) as a $\bm{Z}[G_0]$-module where the action of
$G_0$ is trivial. Define a morphism $\Phi\colon N\to
\langle\zeta\rangle^m$ of $\bm{Z}[G_0]$-modules by $\Phi
\bigl(\sum_{g\in G_0,1\le i\le n} a_{i,g} y_i(g) \bigr)
=\left(\frac{\tau_1(Y)}{Y},\frac{\tau_2(Y)}{Y},\ldots,\frac{\tau_m(Y)}{Y}\right)$
where $Y=\prod_{g\in G_0,1\le i\le n} y_i(g)^{a_{i,g}}$ $\in
k(y_i(g):1\le i\le n, g\in G_0)$.

Define $M=\fn{Ker}(\Phi)$.  Clearly $M$ is a $G_0$-lattice.

It is easy to see that $k(y_i(g):1\le i\le n, g\in G_0)^A=k(M)$,
i.e.\ if $M=\bigoplus_{1\le l\le e} \bm{Z} \cdot z_l$, then
$k(y_i(g):1\le i\le n, g\in G_0)^A=k(z_1,z_2,\ldots,z_e)$ where
each $z_l$ is a monomial in $y_i(g)$'s.

Moreover, $k(y_i(g): 1\le i\le n, g\in G_0)^G=\{k(y_i(g):1\le i\le
n, g\in G_0)^A\}^{G_0}=k(M)^{G_0}$. The group $G_0$ acts on $k(M)$
by purely monomial $k$-automorphisms (see Definition \ref{d3.10}).
Applying Theorem \ref{t3.11}, we find that
$\fn{Br}_{v,k}(k(M)^{G_0})=0$. Hence the result.
\end{proof}

%---------------------------------r
\begin{remark}
Saltman shows that, if $G=A\rtimes G_0$ where $A$ is abelian
normal such that (i) $\gcd\{|A|,|G_0|\}=1$, and (ii) both $k(A)$
and $k(G_0)$ are retract $k$-rational, then $k(G)$ is also retract
$k$-rational \cite[Theorem 3.5; Ka4, Theorem 3.5]{Sa1}.
\end{remark}

Now we turn to groups belonging to the isoclinism family $\Phi_5$ for groups of order $p^5$.

%-------------------------------d4.3
\begin{defn} \label{d4.3}
Let $p$ be an odd prime number.
The isoclinism family $\Phi_5$ for groups of order $p^5$ consists of two groups: $\Phi_5(2111)$ and $\Phi_5(1^5)$ (see \cite[page 620]{Ja}).
These two groups are defined as follows.

For $G=\Phi_5(2111)$, $G=\langle f_i:1\le i\le 5\rangle$ with $Z(G)=\langle f_5\rangle$ and relations
\begin{gather*}
[f_1,f_2]=[f_3,f_4]=f_5,\ [f_1,f_3]=[f_2,f_3]=[f_1,f_4]=[f_2,f_4]=1, \\
f_1^p=f_5,\ f_i^p=1 \text{ for } 2\le i\le 5.
\end{gather*}

For $G=\Phi_5(1^5)$, $G=\langle f_i:1\le i\le 5\rangle$ with $Z(G)=\langle f_5\rangle$ and relations
\begin{gather*}
[f_1,f_2]=[f_3,f_4]=f_5,\ [f_1,f_3]=[f_2,f_3]=[f_1,f_4]=[f_2,f_4]=1,\\
f_i^p=1 \text{ for } 1\le i\le 5.
\end{gather*}
\end{defn}

Note that both $\Phi_5(2111)$ and $\Phi_5(1^5)$ are extra-special
$p$-groups.

%-----------------------------t4.4
\begin{theorem} \label{t4.4}
Let $p$ be an odd prime number and $G$ belong to the isoclinism
family $\Phi_5$ for groups of order $p^5$. Then $B_0(G)=0$.
\end{theorem}

\begin{proof}
Choose an algebraically closed field $k$ with $\fn{char}k\ne p$ (in particular, we may choose $k=\bm{C}$).
If $\fn{Br}_{v,k}(k(G))=0$, then $B_0(G)=0$ by Theorem \ref{t1.4}.
Hence we will show that $\fn{Br}_{v,k}(k(G))=0$ by using Theorem \ref{t4.2}.

For $G=\Phi_5(2111)$ or $\Phi_5(1^5)$, write $G=A\rtimes G_0$ where
$A=\langle f_1,f_3,f_5\rangle$ and $G_0=\langle f_2,f_4\rangle$.
Conditions (i), (ii), (iii) in Theorem \ref{t4.2} are satisfied.
Hence we may apply Theorem \ref{t4.2}. Done.
\end{proof}

Now we consider groups in the isoclinism family $\Phi_7$. Since
the relations for $p=3$ and $p \ge 5$ are not the same (due to the
notation $\alpha_1^{(p)}=1$), we define these groups separately.

%-----------------------d4.5
\begin{defn} \label{d4.5}
Let $p$ be a prime number  and $p\ge 5$. The isoclinism family
$\Phi_7$ for groups of order $p^5$ consists of five groups:
$G=\Phi_7(2111)a$, $\Phi_7(2111)b_1$, $\Phi_7(2111)b_\nu$ (where
$2\le \nu\le p-1$ and $\nu$ is a fixed quadratic non-residue modulo
$p$), $\Phi_7(2111)c$ and $\Phi_7(1^5)$ (see \cite[page~621]{Ja}).
These groups $G$ are defined by $G=\langle f_i: 0\le i\le 4\rangle$
with $Z(G)=\langle f_3\rangle$, common relations
\begin{gather*}
[f_1,f_0]=f_2,\ [f_2,f_0]=[f_1,f_4]=f_3,\ [f_4,f_0]=[f_2,f_1]=[f_4,f_2]=1,
\end{gather*}
but with extra relations
\leftmargini=8ex
\begin{enumerate}
\item[(1)]
for $G=\Phi_7(2111)a: f_0^p=f_3$, $f_i^p=1$ for $1\le i\le 4$;
\item[(2)]
for $G=\Phi_7(2111)b_1: f_1^p=f_3,\ f_0^p=f_i^p=1$ for $2\le i\le 4$;
\item[(3)]
for $G=\Phi_7(2111)b_\nu: f_1^p=f_3^\nu$, $f_0^p=f_i^p=1$ for $2\le i\le 4$;
\item[(4)]
for $G=\Phi_7(2111)c: f_4^p=f_3$, $f_i^p=1$ for $0\le i\le 3$;
\item[(5)]
for $G=\Phi_7(1^5): f_i^p=1$ for $0\le i\le 4$.
\end{enumerate}
\end{defn}

%--------------------------t4.6
\begin{theorem} \label{t4.6}
Let $p$ be a prime number and $p\ge 5$. If $G$ belongs to the
isoclinism family $\Phi_7$ for groups of order $p^5$, then
$B_0(G)=0$.
\end{theorem}

\begin{proof}
The proof is similar to that of Theorem \ref{t4.4} by applying
Theorem \ref{t4.2}. We write $G=A\rtimes G_0$ for suitable
subgroups $A$ and $G_0$. Here are the subgroups we choose.

If $G=\Phi_7(2111)a$ or $\Phi_7(2111)c$, $A=\langle f_0,f_4\rangle$, $G_0=\langle f_1,f_2\rangle$.

If $G=\Phi_7(2111)b_1$ or $\Phi_7(2111)b_\nu$, $A=\langle f_1,f_2\rangle$, $G_0=\langle f_0,f_4\rangle$.

If $G=\Phi_7(1^5)$, $A=\langle f_0, f_3,f_4\rangle$, $G_0=\langle f_1,f_2\rangle$.
\end{proof}

%-------------------------d4.7
\begin{defn} \label{d4.7}
The isoclinism family $\Phi_7$ for groups of order $3^5$ consists of
five groups: $G=G(3^5,i)$ where $56\le i\le 60$ and $G(3^5,i)$ is
the GAP code number. These groups $G$ are defined by $G=\langle f_i:
1\le i\le 5\rangle$ with $Z(G)=\langle f_5\rangle$, common relations
\begin{gather*}
[f_2,f_1]=f_4,\ [f_3,f_2]=[f_4,f_1]=f_5,\ [f_3,f_1]=[f_4,f_2]=[f_4,f_3]=1,
\end{gather*}
but with extra relations
\leftmargini=8ex
\begin{enumerate}
\item[(1)] for $G=G(3^5,56): f_i^3=1$ for $1\le i\le 5$;
\item[(2)] for $G=G(3^5,57): f_2^3=f_5$, $f_1^3=f_i^3=1$ for $3\le
i\le 5$; \item[(3)] for $G=G(3^5,58): f_2^3=f_5^2$,
$f_1^3=f_i^3=1$ for $3\le i\le 5$; \item[(4)] for $G=G(3^5,59):
f_1^3=f_2^{-3}=f_5$, $f_i^3=1$ for $3\le i\le 5$; \item[(5)] for
$G=G(3^5,60): f_3^3=f_5$, $f_1^3=f_2^3=f_4^3=f_5^3=1$.
\end{enumerate}

Note that, in the notation of \cite[page~621]{Ja}, the GAP groups
$G(3^5,i)$, $56\le i\le 60$, correspond to $\Phi_7(2111)b_1$,
$\Phi_7(2111)b_\nu$, $\Phi_7(1^5)$, $\Phi_7(2111)a$ and
$\Phi_7(2111)c$ respectively.
\end{defn}

%----------------------------t4.8
\begin{theorem} \label{t4.8}
If $G$ is a group belonging to the isoclinism family $\Phi_7$ for groups of order $3^5$, then $B_0(G)=0$.
\end{theorem}

\begin{proof}
The proof is the same as that of Theorem \ref{t4.6} except for
$G=G(3^5,59)$. Write $G=A\rtimes G_0$ where $G_0\simeq C_3 \times
C_3$, and \leftmargini=8ex
\begin{enumerate}
\item[(i)]
if $G=G(3^5,56)$, $A=\langle f_2,f_4,f_5\rangle$, $G_0=\langle f_1,f_3\rangle$;
\item[(ii)]
if $G=G(3^5,57)$ or $G(3^5,58)$, $A=\langle f_2,f_4\rangle$, $G_0=\langle f_1,f_3\rangle$;
\item[(iii)]
if $G=G(3^5,60)$, $A=\langle f_1,f_3\rangle$, $G_0=\langle f_2,f_4\rangle$.
\end{enumerate}

It remains to show that $B_0(G)=0$ for $G=G(3^5,59)$.

\bigskip
Step 1. Let $\eta$ be a primitive 9th root of unity and
$\zeta=\eta^3$. We will construct a faithful 9-dimensional
representation of $G$ over $k$, which may be embedded into the
regular representation of $G$. The method is similar to that of
Step 1 in the proof of Theorem \ref{t4.2}.

Let $A=\langle f_1,f_3\rangle=\langle f_1,f_3,f_5\rangle \simeq
C_9\times C_3$ act on the 1-dimensional space $k\cdot X$ by
$f_1\cdot X=\eta X$, $f_3\cdot X=X$. It follows that $f_5\cdot
X=\zeta X$.

The above action defines a faithful linear character $\rho: A\to
k^{\times}$. The induced representation can be written explicitly
as follows.

Define $V=\bigoplus_{1\le i\le 9} k\cdot x_i$ where $x_1=X$, $x_2=f_4\cdot X$, $x_3=f_4^2\cdot X$, $x_4=f_2\cdot X$, $x_5=f_2f_4\cdot X$, $x_6=f_2f_4^2\cdot X$,
$x_7=f_2^2\cdot X$, $x_8=f_2^2 f_4\cdot X$, $x_9=f_2^2f_4^2\cdot X$.
The action of $G$ on $x_i$ is given by
\begin{align*}
f_1:{} & x_1 \mapsto \eta x_1,\ x_2\mapsto \eta^7 x_2,\ x_3\mapsto \eta^4 x_3,\ x_4\mapsto \eta^4 x_6,\ x_5\mapsto \eta x_4,\ x_6\mapsto \eta^7 x_5, \\
& x_7\mapsto \eta^7 x_8,\ x_8\mapsto \eta^4 x_9,\ x_9\mapsto \eta x_7, \\
f_2:{} & x_1\mapsto x_4\mapsto x_7\mapsto \zeta^2x_1,\ x_2\mapsto x_5\mapsto x_8\mapsto \zeta^2 x_2,\ x_3\mapsto x_6\mapsto x_9\mapsto \zeta^2 x_3, \\
f_3:{} & x_1\mapsto x_1,\ x_2\mapsto x_2,\ x_3\mapsto x_3,\ x_4\mapsto \zeta x_4,\ x_5\mapsto \zeta x_5,\ x_6\mapsto \zeta x_6,\\
& x_7\mapsto \zeta^2 x_7,\ x_8\mapsto \zeta^2 x_8,\ x_9\mapsto \zeta^2 x_9, \\
f_4:{} & x_1\mapsto x_2\mapsto x_3\mapsto x_1,\ x_4\mapsto x_5\mapsto x_6\mapsto x_4,\ x_7\mapsto x_8\mapsto x_9\mapsto x_7, \\
f_5:{} & x_i\mapsto \zeta x_i \text{ for }1\le i\le 9.
\end{align*}

By Lemma \ref{l3.9}, it is a faithful representation of $G$.
This representation can be embedded into the regular representation of $G$,
because it is an irreducible representation of $G$.

Apply Theorem \ref{t3.2}.
We find that $k(G)$ is rational over $k(x_i:1\le i\le 9)^G$.

\bigskip
Step 2.
Define $u_1=x_4/x_1$, $u_2=x_7/x_4$, $u_3=x_2/x_1$, $u_4=x_3/x_2$, $u_5=x_5/x_4$, $u_6=x_6/x_5$, $u_7=x_8/x_7$, $u_8=x_9/x_8$.
Apply Theorem \ref{t3.1}.
We find that $k(x_i:1\le i\le 9)^G=k(u_i: 1\le i\le 8)^G(u_0)$ for some element $u_0$ fixed by the action of $G$.

We conclude that $k(G)$ is rational over $k(u_i:1\le i\le 8)^G$.

By Theorem \ref{t3.7} and Theorem \ref{t1.4},
it follows that $B_0(G)\simeq \fn{Br}_{v,k} (k(u_i:1\le i\le 8)^G)$.

\bigskip
Step 3.
Now consider the group $H=G(3^5,58)$.
We will repeat the procedure of Step 1 and Step 2 for $H$.

Namely, define $B=\langle f_1,f_3,f_5\rangle \simeq C_3\times C_3\times C_3$.
Let $B$ act on $k\cdot Y$ by $f_1\cdot Y=f_3\cdot Y=Y$, $f_5\cdot Y=\zeta Y$.

Construct the induced representation $W=\bigoplus_{1\le i\le 9}
k\cdot y_i$ where $y_1=Y$, $y_2=f_4\cdot Y$, $y_3=f_4^2\cdot Y$,
$y_4=f_2\cdot Y$, $y_5=f_2f_4\cdot Y$, $y_6=f_2f_4^2\cdot Y$,
$y_7=f_2^2\cdot Y$, $y_8=f_2^2f_4\cdot Y$, $y_9=f_2^2f_4^2\cdot Y$.
The actions of $f_2$, $f_3$, $f_4$, $f_5$ on $W$ are the same as
those on $V$ (just replace $x_i$'s by $y_i$'s), but
\begin{align*}
f_1:{} & x_1 \mapsto x_1,\ x_2\mapsto \zeta^2x_2,\ x_3\mapsto\zeta x_3,\ x_4\mapsto\zeta x_6,\ x_5\mapsto x_4,\ x_6\mapsto \zeta^2 x_5, \\
& x_7\mapsto \zeta^2 x_8,\ x_8\mapsto \zeta x_9,\ x_9\mapsto x_7.
\end{align*}

The coincidence of the group actions can be explained as follows.
The relations of $G(3^5,59)$ and $G(3^5,58)$ are almost the same except for $f_1^3=f_5$ in $G(3^5,59)$ and $f_1^3=1$ in $G(3^5,58)$.

\bigskip
Step 4. Define $v_1=y_4/y_1$, $v_2=y_7/y_4$, $v_3=y_2/y_1$,
$v_4=y_3/y_2$, $v_5=y_5/y_4$, $v_6=y_6/y_5$, $v_7=y_8/y_7$,
$v_8=y_9/y_8$. Similar to Step 2, we get that $k(H)$ is rational
over $k(v_i:1\le i\le 8)^H$ and $B_0(H)\simeq \fn{Br}_{v,k}
(k(v_i: 1\le i\le 8)^H)$.

Compare the actions of $G$ on $u_1,\ldots, u_8$ with the actions of $H$ on $v_1,\ldots,v_8$.
We find they are the same!

Thus $k(u_i:1\le i\le 8)^G\simeq k(v_i:1\le i\le 8)^G$ over $k$.

Hence $B_0(G)\simeq \fn{Br}_{v,k} (k(u_i:1\le i\le 8)^G)\simeq \fn{Br}_{v,k} (k(v_i:1\le i\le 8)^H) \simeq B_0(H)$.
But $B_0(H)=0$ has been proved at the beginning.
Hence $B_0(G)=0$.
\end{proof}

%-----------------------------------------------------S5
\section{\boldmath $B_0(G)=0$ for the groups belonging to $\Phi_6$}

Let $p$ be an odd prime number.
Throughout this section $g$ is the smallest positive integer which is a primitive root modulo $p$,
and $\nu$ is the smallest positive integer which is a quadratic non-residue modulo $p$.

%-------------------------d5.1
\begin{defn} \label{d5.1}
Let $p$ be an odd prime prime. The isoclinism family $\Phi_6$ for
groups of order $p^5$ consists of the groups $G=\Phi_6(221)a$,
$\Phi_6(221)b_r$ (where $1\le r\le (p-1)/2$), $\Phi_6(221)c_r$
(where $r=1$ or $\nu$), $\Phi_6(221)d_0$, $\Phi_6(221)d_r$ (where
$1\le r\le (p-1)/2$), $\Phi_6(2111)a$ (this group exists only for
$p\ge 5$), $\Phi_6(2111)b_r$ (where $r=1$ or $\nu$; these groups
exist only for $p\ge 5$), and $\Phi_6(1^5)$. When $p\ge 5$, there
are $p+7$ such groups; when $p=3$, there are 7 such groups (see
\cite[pages~620--621]{Ja}). These groups $G$ are defined by
$G=\langle f_1,f_2,f_0,h_1,h_2\rangle$ with $Z(G)=\langle
h_1,h_2\rangle$, common relations
\begin{gather*}
[f_1,f_2]=f_0,\ [f_0,f_1]=h_1,\ [f_0,f_2]=h_2,\ f_0^p=h_1^p=h_2^p=1,
\end{gather*}
but with extra relations
\leftmargini=8ex
\begin{enumerate}
\item[(1)]
for $G=\Phi_6(221)a: f_1^p=h_1$, $f_2^p=h_2$;
\item[(2)]
for $G=\Phi_6(221)b_r: f_1^p=h_1^k$, $f_2^p=h_2$ where $k=g^r$;
\item[(3)]
for $G=\Phi_6(221)c_r: f_1^p=h_2^{-r/4}$, $f_2^p=h_1^rh_2^r$;
\item[(4)]
for $G=\Phi_6(221)d_0: f_1^p=h_2$, $f_2^p=h_1^\nu$;
\item[(5)]
for $G=\Phi_6(221)d_r: f_1^p=h_2^k$, $f_2^p=h_1h_2$ where $4k=g^{2r+1}-1$;
\item[(6)]
for $G=\Phi_6(2111)a: f_1^p=h_1$, $f_2^p=1$;
\item[(7)]
for $G=\Phi_6(2111)b_r: f_1^p=1$, $f_2^p=h_1^r$;
\item[(8)]
for $G=\Phi_6(1^5): f_1^p=1$, $f_2^p=1$.
\end{enumerate}
\end{defn}

(Note that whenever the exponent of $h_2$ is fractional, it is
understood that it is taken modulo $p$, which is the order of
$h_2$.)

\bigskip
Before proving $B_0(G)=0$ for the groups $G$ in Definition
\ref{d5.1}, we recall two results in group cohomology.

%--------------------------t5.4
\begin{theorem}[Dekimpe, Hartl and Wauters \cite{DHW}] \label{t5.4}
Let $G$ be a finite group, $N$ a normal subgroup of $G$. Then the
Hochschild--Serre spectral sequence gives rise to the following
$7$-term exact sequence
\begin{align*}
0 &\to H^1(G/N,\bm{Q}/\bm{Z}) \to H^1(G,\bm{Q}/\bm{Z})\to H^1(N,\bm{Q}/\bm{Z})^G \to H^2(G/N,\bm{Q}/\bm{Z}) \\
&\to H^2(G,\bm{Q}/\bm{Z})_1 \to H^1(G/N,H^1(N,\bm{Q}/\bm{Z}))
\xrightarrow{\lambda} H^3(G/N,\bm{Q}/\bm{Z})
\end{align*}
where $H^2(G,\bm{Q}/\bm{Z})_1=\fn{Ker}\{H^2(G,\bm{Q}/\bm{Z})
\xrightarrow{\fn{res}} H^2(N,\bm{Q}/\bm{Z})\}$ and $\lambda$ is
defined as follows. Choose a section $u\colon G/N\to G$ and define a
$2$-cocycle $\varepsilon\colon G/N\times G/N\to N$ satisfying
$u(\tau)u(\tau')=\varepsilon(\tau,\tau')u(\tau\tau')$ for any
$\tau,\tau'\in G/N$. For each $1$-cocycle $\gamma\colon G/N\to
H^1(N,\bm{Q}/\bm{Z})$, the map $\lambda$ is defined by
\[
\xymatrix@R=4pt{\lambda\colon H^1(G/N, H^1(N,\bm{Q}/\bm{Z})) \ar[r] & H^3(G/N,\bm{Q}/\bm{Z}) \\
\gamma \ar@{|->}[r] & \lambda (\gamma)=c}
\]
where $c: G/N\times G/N \times G/N \to \bm{Q}/\bm{Z}$ is the
$3$-cocycle defined as
$c(\tau_1,\tau_2,\tau_3)=(\side{u(\tau_1\tau_2)}{\gamma(\tau_3)})
(\varepsilon(\tau_1,\tau_2))$ for all $\tau_1,\tau_2,\tau_3\in
G/N$.
\end{theorem}

\begin{proof}
See \cite{DHW} for details.

The formula for $\lambda$ is summarized in \cite[page~21, formula~
(6)]{DHW}. If $\gamma\colon G/N \to H^1(N,M)$ is a 1-cocycle where
$M$ is a $G$-module, $[\gamma]$ denotes its cohomology class in
$H^1(G/N,H^1(N,M))$ in the paper \cite{DHW}. The image
$\lambda([\gamma])\in H^3(G/N,M^N)$ is represented by a 3-cocycle
$c\colon G/N\times G/N \times G/N \to M^N$ which is given on
\cite[page 21]{DHW}. Note that the definition of $-\delta^0\colon
M\to \fn{Der}(N,M)$ can be found on \cite[page~14]{DHW}.

When $M$ is a trivial $G$-module, $-\delta^0$ is a zero map and
therefore the map $F'\colon G/N \times G/N \to M$ on \cite[page~
21]{DHW} can be chosen to be a zero map. Consequently,
$c(q_1,q_2,q_3)=(\side{s_1(q_1q_2)}{s_2 D(q_3)})(F_1(q_1,q_2))$ for
any $q_1,q_2,q_3\in G/N$. This is our formula when
$M=\bm{Q}/\bm{Z}$.
\end{proof}

%---------------------------t5.5
\begin{theorem} \label{t5.5}
Let $p$ be a prime number, $C_p=\langle \sigma \rangle$ and $M$ be a
$C_p$-module. For any $1$-cocycle $\beta\colon C_p\to M$, the
following map
\[
\xymatrix@R=4pt{\Phi: H^1(C_p,M) \ar[r] & H^3(C_p,M) \\ \beta
\ar@{|->}[r] & \Phi(\beta)=\gamma}
\]
is a group isomorphism where $\gamma: C_p\times C_p\times C_p\to
M$ is a $3$-cocycle defined as
\[
\gamma(\sigma^i,\sigma^j,\sigma^l)=\begin{cases}
0 & \text{if } 0\le i+j\le p-1 \\
\left(\side{\sigma^{i+j}}{\beta}\right) (\sigma^l) & \text{if
}i+j\ge p
\end{cases}
\]
where $0\le i,j,l\le p-1$.
\end{theorem}

\begin{proof}
By \cite[page 149, Theorem 14]{Se}, the 2-cocycle $\alpha:
C_p\times C_p\to \bm{Z}$ defined as
\[
\alpha(\sigma^i,\sigma^j)=\begin{cases} 0 & \text{if }0\le i+j\le
p-1 \\ 1 & \text{if } i+j\ge p \end{cases}
\]
represents a ``fundamental" cohomology class in $H^2(C_p,\bm{Z})$
such that, for any $C_p$-module $M$, the map
\[
\xymatrix@R=4pt{\Phi: H^1(C_p,M) \ar[r] & H^3(C_p,M) \\ \beta
\ar@{|->}[r] & \Phi(\beta)=\alpha\cup \beta }
\]
is an isomorphism where $\alpha \cup \beta$ is the cup product. It
is easy to check that $\alpha \cup \beta =\gamma$ where $\gamma$
is defined in the statement of this theorem.
\end{proof}

%-------------------------t5.6
\begin{theorem} \label{t5.6}
Let $G$ be the group $\Phi_6(221)a$ in Definition \ref{d5.1}. Then
$B_0(G)=0$.
\end{theorem}

\begin{proof}
Step 1. Write $G=\langle f_1,f_2,f_0,h_1,h_2\rangle$. Choose
$N=\langle f_1,f_0,h_1,h_2\rangle$; $N$ is a normal subgroup of
$G$. We will apply Theorem \ref{t5.4} to the group extension $1\to
N\to G\to G/N\to 1$.

Since $G/N=\langle \bar{f}_2 \rangle \simeq C_p$, we find that
$H^2(G/N,\bm{Q}/\bm{Z})=0$ \cite[page~37, Corollary~2.2.12]{Kar}).
By Theorem \ref{t5.4}, we obtain the following exact sequence
\[
0\to H^2(G,\bm{Q}/\bm{Z})_1 \to H^1(G/N,H^1(N,\bm{Q}/\bm{Z}))
\xrightarrow{\lambda} H^3(G/N,\bm{Q}/\bm{Z}).
\]

\bigskip
Step 2. Note that $B_0(G)$ is a subgroup of
$H^2(G,\bm{Q}/\bm{Z})_1$.

For, consider the restriction map $\fn{res}:H^2(G,\bm{Q}/\bm{Z})
\to H^2(N,\bm{Q}/\bm{Z})$. It induces a map $\fn{res}: B_0(G)\to
B_0(N)$ such that the following diagram commutes
\[
\xymatrix{B_0(G) \ar[r] \ar[d] & B_0(N) \ar[d] \\
H^2(G,\bm{Q}/\bm{Z}) \ar[r]_{\fn{res}} & H^2(N,\bm{Q}/\bm{Z}).}
\]

Since $N$ is a $p$-group of order $p^4$, $k(N)$ is $k$-rational
for any algebraically closed field $k$ with $\fn{char}k\ne p$ by
Theorem \ref{t1.6}. It follows that $B_0(N)\simeq \fn{Br}_{v,k}
(k(N))=0$ by Lemma \ref{l1.3} and Theorem \ref{t1.4}. Hence
$B_0(G)$ is contained in the kernel of the map $\fn{res}:
H^2(G,\bm{Q}/\bm{Z})\to H^2(N,\bm{Q}/\bm{Z})$. That is, $B_0(G)$
is a subgroup of $H^2(G,\bm{Q}/\bm{Z})_1$.

If we can show that $H^2(G,\bm{Q}/\bm{Z})_1=0$, then $B_0(G)=0$
and the proof is finished. Note that $H^2(G,\bm{Q}/\bm{Z})_1=0$ if
and only if $\lambda$ is an injective map by the exact sequence in
Step 1.

\bigskip
Step 3. We recall a general fact about $H^1(C_n,M)$.

 Let $G=\langle \sigma \rangle
\simeq C_n$ and $M$ be a $G$-module. Define the map $Norm: M \to
M$ by $Norm \ (x)=x + \sigma \cdot x + \sigma^2 \cdot x + \cdots
+\sigma^{n-1} \cdot x$ for any $x \in M$. It is well-known that
$H^1(G,M) \simeq \fn{Ker}(Norm)/\fn{Image}(\sigma -1)$. We will
give an explicit correspondence between these two groups. If $x
\in M$ satisfies $Norm \ (x)=0$, define a normalized $1$-cocycle
$\beta_x : G \to M$ by $\beta_x(\sigma)=x,\beta_x(\sigma^i)=x+
\sigma \cdot x + \sigma^2 \cdot x + \cdots +\sigma^{i-1} \cdot x$
for $0\le i \le n-1$. It is easy to see that $x \in \fn{Image}
(\sigma -1)$ if and only $\beta_x$ is cohomologously trivial.

\bigskip
Step 4. We will determine $H^1(G/N,H^1(N,\bm{Q}/\bm{Z}))$.

To keep the notations clean and transparent, we adopt the
multiplicative notation for $\bm{Q}/\bm{Z}$, i.e.\ we identify
$\bm{Q}/\bm{Z}$ with all the roots of unity in $\bm{C} \backslash
\{0\}$. Thus a primitive $p$-th root of unity is the element $i/p$
(for some $1\le i\le p-1$) in the additive notation of
$\bm{Q}/\bm{Z}$.

Let $\zeta$ be a primitive $p$-th root of unity. Since
$H^1(N,\bm{Q}/\bm{Z}) \simeq \fn{Hom}(N,\bm{Q}/\bm{Z}) \simeq
\fn{Hom}(N/[N,N],\bm{Q}/\bm{Z})$ and $N/[N,N]=\langle
\bar{f}_1,\bar{f}_0,\bar{h}_2\rangle \simeq C_p\times C_p\times
C_p$, we find that $H^1(N,\bm{Q}/\bm{Z})=\langle
\varphi_1,\varphi_0,\psi\rangle$ where these 1-cocycles
$\varphi_1$, $\varphi_0$, $\psi$ are defined as
\begin{alignat*}{2}
\varphi_1(f_1) &=\zeta, &\quad \varphi_1(f_0) &=\varphi_1(h_1)=\varphi_1(h_2)=1, \\
\varphi_0(f_0) &=\zeta, & \varphi_0(f_1) &=\varphi_0(h_1)=\varphi_0(h_2)=1, \\
\psi(h_2) &=\zeta, & \psi(f_1) &=\psi(f_0)=\psi(h_1)=1.
\end{alignat*}

The group $G$ (resp. $G/N=\langle \bar{f}_2\rangle$) acts on
$H^1(N,\bm{Q}/\bm{Z})=\langle \varphi_1,\varphi_0,\psi\rangle$. It
is easy to verify that
\[
\side{\bar{f}_2}{\varphi_1}=\varphi_1, \quad
\side{\bar{f}_2}{\varphi_0}=\varphi_1\varphi_0,\quad
\side{\bar{f}_2}{\psi}=\varphi_0\psi.
\]

Consider the norm map $Norm\colon H^1(N,\bm{Q}/\bm{Z}) \to
H^1(N,\bm{Q}/\bm{Z})$ defined by the action of $\bar{f}_2$ (see Step
3).

We find that $H^1(G/N,H^1(N,\bm{Q}/\bm{Z}))\simeq
\fn{Ker}(Norm)/\fn{Image}(\bar{f}_2-1) =\langle
\varphi_1,\varphi_0,\psi\rangle/\langle\varphi_1,$
$\varphi_0\rangle$ if $p\ge 5$. But, if $p=3$,
$\fn{Ker}(1+\bar{f}_2+\bar{f}_2^2)=\langle
\varphi_1,\varphi_0\rangle$.

It follows that
\[
H^1(G/N,H^1(N,\bm{Q}/\bm{Z})) =\begin{cases} 0, & \text{if } p=3;
\\ \langle \bar{\psi} \rangle \simeq C_p, & \text{if }p\ge 5.
\end{cases}
\]

When $p=3$, we obtain $H^2(G,\bm{Q}/\bm{Z})_1=0$ from the exact
sequence in Step 1. Hence $B_0(G)=0$.

From now on, we assume that $p\ge 5$. By Step 3, the element
$\bar{\psi} \in \fn{Ker}(Norm)/\fn{Image}$ $(\bar{f}_2-1)$
corresponds to the 1-cocycle $\beta: G/N \to H^1(N,\bm{Q}/\bm{Z})$
defined as
\begin{align*}
& \beta(1)=1, \quad \beta(\bar{f}_2)=\psi, \\
&\beta(\bar{f}_2^i)=\left(\side{\bar{f}_2}{\beta(\bar{f}_2^{i-1})}\right)\bar{f}_2
=\varphi_1^{\binom{i}{3}} \varphi_0^{\binom{i}{2}} \psi^i
\end{align*}
where $1\le i\le p-1$ and $\binom{a}{b}$ is the binomial
coefficient with the convention that $\binom{a}{b}=0$ if $1 \le a
< b$.

\bigskip
Step 5. We will show that $\lambda(\beta)\ne 0$ and finish the
proof of $B_0(G)=0$.

Follow the description of $\lambda$ in Theorem \ref{t5.4}. Choose
a section $u: G/N \to G$ by $u(1)=1$, $u(\bar{f}_2^i)=f_2^i$ for
$1\le i\le p-1$. It is easy to find the 2-cocycle $\varepsilon:
G/N\times G/N \to N$. In fact, if $0\le i,j\le p-1$, then
\[
\varepsilon(\bar{f}_2^i,\bar{f}_2^j)=\begin{cases} 1, & \text{if }
0\le i+j\le p-1; \\ h_2, & \text{if }i+j\ge p; \end{cases}
\]
the second alternative follows from the fact $f_2^p=h_2$.

Now we will evaluate $\lambda(\beta)$ where $\beta$ is the
1-cocycle determined in Step 4. Write $c=\lambda(\beta)$. Then,
for $0\le i,j,l\le p-1$,
\[
c(\bar{f}_2^i,\bar{f}_2^j,\bar{f}_2^l)=\left(\side{u(\bar{f}_2^{i+j})}{\beta(\bar{f}_2^l)}\right)(\varepsilon(\bar{f}_2^i,\bar{f}_2^j))
\]
by Theorem \ref{t5.4}.

In particular, for $0\le i\le p-1$, we have
\begin{align*}
c(\bar{f}_2,\bar{f}_2^{p-1},\bar{f}_2^i) &= \left(\side{u(1)}{\beta(\bar{f}_2^i)}\right) (\varepsilon(\bar{f}_2,\bar{f}_2^{p-1})) =(\beta(\bar{f}_2^i))(h_2) \\
&= \left(\varphi_1^{\binom{i}{3}} \varphi_0^{\binom{i}{2}}
\psi^i\right)(h_2)=(\psi(h_2))^i=\zeta^i.
\end{align*}

On the other hand, apply Theorem \ref{t5.5} for
$\Phi:H^1(G/N,\bm{Q}/\bm{Z})\to H^3(G/N,\bm{Q}/\bm{Z})$. We will
find a 1-cocycle $\tilde{\beta}:G/N\to \bm{Q}/\bm{Z}$ such that
$\Phi(\tilde{\beta})=c\in H^3(G/N,\bm{Q}/\bm{Z})$. In fact, from
Theorem \ref{t5.5},
$c(\bar{f}_2,\bar{f}_2^{p-1},\bar{f}_2^i)=\tilde{\beta}(\bar{f}_2^i)$.
Thus $\tilde{\beta}(\bar{f}_2^i)=\zeta^i$ for all $0\le i\le p-1$.

By Step 3, the 1-cocycle $\tilde{\beta}\in H^1(G/N,\bm{Q}/\bm{Z})$
corresponds to the non-zero element $\bar{\zeta} \in
\fn{Ker}(Norm)/\fn{Image}(\bar{f}_2-1)$, regarding $\zeta$ as an
element in $\fn{Ker}(Norm)$ where $Norm : \bm{Q}/\bm{Z} \to
\bm{Q}/\bm{Z}$ is defined by the action of $\bar{f}_2$ (see Step
3). Hence $\tilde{\beta}\ne 0$ and $\Phi(\tilde{\beta})=c\ne 0$.
Thus $\lambda$ is injective.
\end{proof}

\bigskip
 The proof of the following lemma is routine and is omitted.

%-----------------------------l5.2
\begin{lemma} \label{l5.2}
Let $G$ be a group in Definition \ref{d5.1}. If $0\le i,j\le p-1$,
then $f_0^j f_1^i=f_1^if_0^jh_1^{ij}$,
$f_0^jf_2^i=f_2^if_0^jh_2^{ij}$, and
\[
f_2^if_1^j=f_1^jf_2^if_0^{-ij}h_1^{-i\binom{j}{2}} h_2^{-j\binom{i}{2}}.
\]
\end{lemma}

%---------------------------t5.3
\begin{theorem} \label{t5.3}
Let $p$ be an odd prime number. If $G$ is a group belonging to the
isoclinism family $\Phi_6$ for groups of order $p^5$, then
$B_0(G)=0$.
\end{theorem}

\begin{proof}
Let $k$ be an algebraically closed field with char$k \neq p$ (in
particular, we may choose $k=\bm{C}$). Let $\eta \in k$ be a
primitive $p^2$-th root of unity and $\zeta=\eta^p$. In the
following we adopt the notation in Definition \ref{d5.1}. We will
show that the fields $k(G)$ are isomorphic to one another over $k$
for all groups $G$ in the isoclinism family $\Phi_6$. Thus they have
isomorphic $Br_{v,k}(k(G)) \simeq B_0(G)$ by Theorem \ref{t1.4}.
Since $B_0(G)=0$ if $G=\Phi_6(221)a$ by Theorem \ref{t5.6}, it
follows that $B_0(G)=0$ for all other groups $G$.

\begin{Case}{1} $G=\Phi_6(221)a$, $\Phi_6(221)b_r$ (where $1\le r\le (p-1)/2$), $\Phi_6(2111)a$, $\Phi_6(1^5)$. \end{Case}

Step 1. For these groups $G$, we have
\[
f_1^p=h_1^{e_1}, \quad f_2^p=h_2^{e_2}
\]
where $0\le e_1,e_2\le p-1$.

We will employ the same method as in Step 1 of the proof in
Theorem \ref{t4.8}.

Consider the subgroups $H_1=\langle f_1,f_0,h_1,h_2\rangle$ and $H_2=\langle f_2,f_0,h_1,h_2\rangle$ of $G$.
Note that $H_2=\langle f_2,f_0,h_2\rangle \times \langle h_1\rangle \simeq \langle f_2,f_0,h_2\rangle \times C_p$.
Hence we get a linear character of $H_2$ so that $\langle f_2,f_0,h_2\rangle$ is the kernel.
Explicitly, we may define an action of $H_2$ on $k\cdot X$ defined by
\[
h_1\cdot X=\zeta X,\quad f_2\cdot X=f_0\cdot X=h_2\cdot X=X.
\]

Similarly, define an action of $H_1$ on $k\cdot Y$ by
\[
h_2\cdot Y=\zeta Y,\quad f_1\cdot Y=f_0\cdot Y=h_1\cdot Y=Y.
\]

Construct the induced representations of these linear characters by
defining $x_i=f_1^i\cdot X$, $y_i=f_2^i\cdot Y$ for $0\le i\le p-1$.
Thus we get an action of $G$ on $(\bigoplus_{0\le i\le p-1} k\cdot
x_i) \oplus (\bigoplus_{0\le i\le p-1} k\cdot y_i)$. With the aid of
Lemma \ref{l5.2}, the action of $G$ is given as follows.
\begin{align*}
f_1 &: x_0\mapsto x_1\mapsto \cdots \mapsto x_{p-1}\mapsto \zeta^{e_1}x_0,\ y_i\mapsto \zeta^{\binom{i}{2}} y_i, \\
f_2 &: x_i\mapsto \zeta^{-\binom{i}{2}}x_i,\ y_0\mapsto y_1\mapsto\cdots \mapsto y_{p-1}\mapsto \zeta^{e_2}y_0, \\
f_0 &: x_i\mapsto \zeta^i x_i,\ y_i\mapsto \zeta^i y_i, \\
h_1 &: x_i\mapsto \zeta x_i,\ y_i\mapsto y_i, \\
h_2 &: x_i\mapsto x_i,\ y_i\mapsto \zeta y_i.
\end{align*}

By Lemma \ref{l3.9}, $G$ acts faithfully on $(\bigoplus_{0\le i\le
p-1} k\cdot x_i)\oplus (\bigoplus_{0\le i\le p-1} k\cdot y_i)$.
Moreover, this representation may be embedded into the regular
representation of $G$. By Theorem \ref{t3.2}, we find that $k(G)$
is rational over $k(x_i,y_i:0\le i\le p-1)^G$.

\bigskip
Step 2. We will apply Theorem \ref{t3.1} to $k(x_i,y_i:0\le i\le
p-1)^G$. Define $u_i=x_i/x_{i-1}$, $U_i=y_i/y_{i-1}$ for $1\le i\le
p-1$. By applying Theorem \ref{t3.1} twice, we get $k(x_i,y_i:0\le
i\le p-1)^G=k(u_i,U_i:1\le i\le p-1)^G(u_0,U_0)$ where $u_0$, $U_0$
are fixed by the action of $G$. The action of $G$ on $u_i$, $U_i$ is
given by
\begin{align*}
f_1 &: u_1\mapsto u_2\mapsto \cdots \mapsto u_{p-1} \mapsto \zeta^{e_1}/(u_1u_2\cdots u_{p-1}),\ U_i\mapsto \zeta^{i-1}U_i, \\
f_2 &: u_i\mapsto \zeta^{-(i-1)}u_i,\ U_1\mapsto U_2\mapsto \cdots\mapsto U_{p-1} \mapsto \zeta^{e_2}/(U_1U_2\cdots U_{p-1}), \\
f_0 &: u_i\mapsto \zeta u_i,\ U_i\mapsto \zeta U_i.
\end{align*}

Note that $h_1(u_i)=h_2(u_i)=u_i$, $h_1(U_i)=h_2(U_i)=U_i$ for $1\le i\le p-1$.  Thus
\begin{align*}
k(u_i,U_i:1\le i\le p-1)^G &= k(u_i,U_i: 1\le i\le p-1)^{G/\langle h_1,h_2\rangle} \\
&=k(u_i,U_i:1\le i\le p-1)^{\langle f_0,f_1,f_2\rangle}.
\end{align*}

\bigskip
Step 3.
Define $u'_i=u_i/\eta^{e_1}$, $U'_i=U_i/\eta^{e_2}$ for $1\le i\le p-1$.

It follows that $k(u_i,U_i:1\le i\le p-1)=k(u'_i,U'_i:1\le i\le p-1)$ and
\begin{align*}
f_1 &: u'_1\mapsto u'_2\mapsto \cdots \mapsto u'_{p-1}\mapsto 1/(u'_1u'_2\cdots u'_{p-1}),\ U'_i\mapsto \zeta^{i-1}U'_i, \\
f_2 &: u'_i\mapsto \zeta^{-(i-1)}u'_i,\ U'_1\mapsto U'_2\mapsto\cdots \mapsto U'_{p-1}\mapsto 1/(U'_1U'_2\cdots U'_{p-1}), \\
f_0 &: u'_i\mapsto \zeta u'_i,\ U'_i\mapsto \zeta U'_i.
\end{align*}

Note that the parameters $e_1$, $e_2$ of these groups $G$ disappear
in the above action. In conclusion, for any group $G$ in this case,
$k(G)$ is rational over $k(u'_i,U'_i: 1 \le i \le p-1)^{\langle
f_1,f_2,f_0 \rangle}$. Thus all these fields $k(G)$ are isomorphic.

\bigskip
\begin{Case}{2} $G=\Phi_6(221)c_r$ (where $r=1$ or $\nu$), $\Phi_6(221)d_r$ (where $1\le r\le (p-1)/2$). \end{Case}

For these groups $G$, we have
\[
f_1^p=h_2^{e_1},\quad f_2^p=h_1^{e_2}h_2^{e_2}
\]
where $1\le e_1,e_2\le p-1$.
The proof is similar to Step 1 and Step 2 of Case 1.

Find integers $e'_1$, $e'_2$ such that $1\le e'_1,e'_2\le p-1$ and $e_1e'_1\equiv e_2e'_2\equiv 1$ (mod $p$).

Consider the subgroups $H_1=\langle f_1,f_0,h_1,h_2\rangle$,
$H_2=\langle f_2,f_0,h_1,h_2\rangle$ of $G$. Since $H_2/\langle
h_2\rangle=\langle \bar{f}_2,\bar{f}_0\rangle \simeq C_{p^2}\times
C_p$, we get a linear character of $H_2$. Similarly for $H_1$.
More precisely, we have actions of $H_2$ on $k\cdot X$, and $H_1$
on $k\cdot Y$ defined by
\begin{alignat*}{3}
f_2\cdot X &= \eta^{e'_2} X, &\quad h_1\cdot X &=\zeta X, &\quad f_0\cdot X &= h_2\cdot X=X, \\
f_1\cdot Y &= \eta^{e'_1}Y, & h_2\cdot Y &= \zeta Y, & f_0\cdot Y &= h_1\cdot Y=Y.
\end{alignat*}

Find the induced representations of $G$ from these two linear characters.
Define $x_i=f_1^i\cdot X$, $y_i=f_2^i\cdot Y$ where $0\le i\le p-1$.
Then $G$ acts faithfully on $(\bigoplus_{0\le i\le p-1} k\cdot x_i)\oplus (\bigoplus_{0\le i\le p-1} k\cdot y_i)$.
Thus $k(G)$ is rational over $k(x_i,y_i:1\le i\le p-1)^G$.

The action of $G$ is given by
\begin{align*}
f_1 &: x_0\mapsto x_1\mapsto \cdots \mapsto x_{p-1}\mapsto x_0,\ y_i\mapsto \eta^{e'_1+p\binom{i}{2}} y_i, \\
f_2 &: x_i\mapsto \eta^{e'_2-p\binom{i}{2}}x_i,\ y_0\mapsto y_1\mapsto \cdots \mapsto y_{p-1} \mapsto \zeta^{e_2}y_0, \\
f_0 &: x_i\mapsto \zeta^ix_i,\ y_i\mapsto \zeta^i y_i, \\
h_1 &: x_i\mapsto \zeta x_i,\ y_i\mapsto y_i, \\
h_2 &: x_i\mapsto x_i,\ y_i\mapsto \zeta y_i.
\end{align*}

Define $u_i=x_i/x_{i-1}$, $U_i=y_i/y_{i-1}$ for $1\le i\le p-1$.
We get $k(x_i,y_i:1\le i\le p-1)^G=k(u_i,U_i:1\le i\le
p-1)^G(u_0,U_0)$ where $u_0$, $U_0$ are fixed by $G$ by applying
Theorem \ref{t3.1} twice. The action of $G$ is given by
\begin{align*}
f_1 &: u_1\mapsto u_2\mapsto \cdots\mapsto u_{p-1}\mapsto 1/(u_1u_2\cdots u_{p-1}),\ U_i\mapsto \zeta^{i-1}U_i, \\
f_2 &: u_i\mapsto \zeta^{-(i-1)}u_i,\ U_1\mapsto U_2\mapsto \cdots\mapsto U_{p-1}\mapsto \zeta^{e_2}/(U_1U_2\cdots U_{p-1}), \\
f_0 &: u_i\mapsto \zeta u_i,\ U_i\mapsto \zeta U_i.
\end{align*}

But the above action is just a special case of the action in Step
2 of Case 1. Hence the result.

\bigskip
\begin{Case}{3} $G=\Phi_6(221)d_0$. \end{Case}

This group satisfies
\[
f_1^p=h_2, \quad f_2^p=h_1^e
\]
where $1\le e\le p-1$. In fact, $e=\nu$.

The proof is the same as for Case 2. Choose an integer $e'$ such
that $1\le e'\le p-1$ and $ee'\equiv 1$ (mod $p$).

\medskip
Consider the subgroups $H_1=\langle f_1,f_0,h_1,h_2\rangle$, $H_2=\langle f_2,f_0,h_1,h_2\rangle$.
Note that $H_1/\langle h_1\rangle \simeq C_{p^2} \times C_p\simeq H_2/\langle h_2\rangle$.
Hence we get vectors $X$ and $Y$ such that
\begin{alignat*}{3}
f_2\cdot X &= \eta^{e'}X, &\quad h_1\cdot X &= X, &\quad f_0\cdot X &= h_2\cdot X=X, \\
f_1\cdot Y &= \eta Y, & h_2\cdot Y&=\zeta Y, & f_0\cdot Y &=h_1\cdot Y=Y.
\end{alignat*}

Construct the induced representation of $G$ on $(\bigoplus_{0\le i\le p-1}k\cdot x_i)\oplus (\bigoplus_{0\le i\le p-1} k\cdot y_i)$ where
$x_i=f_1^i\cdot X$, $y_i=f_2^i\cdot Y$ with $0\le i\le p-1$.
It follows that
\begin{align*}
f_1 &: x_0\mapsto x_1\mapsto \cdots \mapsto x_{p-1} \mapsto x_0,\ y_i\mapsto \eta^{1+p\binom{i}{2}}y_i, \\
f_2 &: x_i\mapsto \eta^{e'-p\binom{i}{2}}x_i,\ y_0\mapsto y_1\mapsto \cdots \mapsto y_{p-1}\mapsto y_0, \\
f_0 &: x_i\mapsto \zeta^i x_i,\ y_i\mapsto \zeta^i y_i, \\
h_1 &: x_i\mapsto \zeta x_i,\ y_i\mapsto y_i, \\
h_2 &: x_i\mapsto x_i,\ y_i\mapsto \zeta y_i.
\end{align*}

By the same arguments as in Case 2, we solve this case.

\bigskip
\begin{Case}{4} $G=\Phi_6(2111)b_r$ (where $r=1$ or $\nu$). \end{Case}

These two groups $G$ satisfy
\[
f_1^p=1 \quad \text{and}\quad f_2^p=h_1^e
\]
where $1\le e\le p-1$.

Choose an integer $e'$ such that $1\le e' \le p-1$ and $ee'\equiv 1$ (mod $p$).

The proof is almost the same as for Case 2.

\medskip
Consider $H_1=\langle f_1,f_0,h_1,h_2\rangle$ and $H_2=\langle f_2,f_0,h_1,h_2\rangle$.
Note that $H_1/\langle h_1\rangle \simeq C_p\times C_p \times C_p$ and $H_2/\langle h_2\rangle \simeq C_{p^2}\times C_p$.
Thus we get vectors $X$ and $Y$ such that
\begin{gather*}
f_2\cdot X=\eta^{e'}X, \quad h_1\cdot X=\zeta X, \quad f_0\cdot X=h_2\cdot X=X, \\
h_2\cdot Y=\zeta Y,\quad f_1\cdot Y=f_0\cdot Y=h_1\cdot Y=Y.
\end{gather*}

Define $x_i=f_1^i\cdot X$, $y_i=f_2^i\cdot Y$ for $0\le i\le p-1$.
The action of $G$ on $k(x_i,y_i:0\le i\le p-1)$ is given by
\begin{align*}
f_1 &: x_0\mapsto x_1\mapsto \cdots \mapsto x_{p-1}\mapsto x_0,\ y_i\mapsto \zeta^{\binom{i}{2}} y_i, \\
f_2 &: x_i\mapsto \eta^{e'-p\binom{i}{2}} x_i,\ y_0\mapsto y_1\mapsto \cdots \mapsto y_{p-1} \mapsto y_0, \\
f_0 &: x_i\mapsto \zeta^i x_i,\ y_i\mapsto \zeta^i y_i, \\
h_1 &: x_i\mapsto \zeta x_i,\ y_i\mapsto y_i, \\
h_2 &: x_i\mapsto x_i,\ y_i\mapsto \zeta y_i.
\end{align*}

The remaining part is the same as in Case 2. Hence the result.
\end{proof}

\bigskip
\begin{proof}[Proof of Theorem \ref{t1.11}]
Combine Theorems \ref{t2.3}, \ref{t4.1}, \ref{t4.4}, \ref{t4.6},
\ref{t4.8} and \ref{t5.3}.
\end{proof}

%---------------------------t5.7
\begin{theorem} \label{t5.7}
Let $p$ be an odd prime number and $k$ be an algebraically closed
field with char $k \neq p$. If $G$ is a group belonging to the
isoclinism family $\Phi_{10}$ for groups of order $p^5$, then
there is a linear representation $G \to GL(V)$ over $k$ satisfying
{\rm (i)} $dim_k V = p^2$, and {\rm (ii)} $k(V)^G$ is not
$k$-rational. In particular, the quotient variety $\bm{P}(V)/G$ is
not $k$-rational where $\bm{P}(V)$ is the projective space
associated to $V$ and the action of $G$ on $\bm{P}(V)$ by
projective linear automorphisms is induced from the linear
representation $G \to GL(V)$.

On the other hand, if $k$ is an algebraically closed field with
char$k \neq 2$ and $G$ is a group belonging to the $16$th
isoclinism family for groups of order $64$, then there is a linear
representation $G \to GL(V)$ over $k$ satisfying {\rm (i)} $dim_k
V = 8$, and {\rm (ii)} $k(V)^G$ is not $k$-rational. In
particular, the quotient variety $\bm{P}(V)/G$ is not
$k$-rational.
\end{theorem}

\begin{proof}
We will find a faithful representation of the required degree for
the group $G$.

In the first case, when $p$ is odd and $G$ is the group given in
the theorem, look into the proof of Theorem \ref{t2.3} for the
generators and relations of $G$. The center of $G$ is $\langle f_5
\rangle$. Take $H= \langle f_3,f_4,f_5 \rangle$; $H$ is an abelian
group. Choose a linear character $\chi : H \to k^{\times}$ such
that $\chi(f_5)=\zeta_p$ and $\chi(f_3)=\chi(f_4)=1$. Designate
the induced representation of $\chi$ (from $H$ to $G$) by $G \to
GL(V)$. It is of degree $p^2$ and is faithful by Lemma \ref{l3.9}.
Note that $Br_{v,k}(k(V)^G)$ is isomorphic to $Br_{v,k}(k(G))$ by
the same arguments as in the proof of Theorem \ref{t4.2}.

For the projective variety $\bm{P}(V)/G$, we use Theorem
\ref{t3.1} and Lemma \ref{l1.3}. In fact, if $k(V)=k(x_i: 1 \le i
\le p^2)$, then $k(\bm{P}(V))=k(x_i/x_1: 2 \le i \le p^2)$. By
Theorem \ref{t3.1}, $k(x_i: 1 \le i \le p^2)^G=k(x_i/x_1: 2 \le i
\le p^2)^G(x)$ for some element $x$. These two fixed fields have
isomorphic unramified Brauer groups by Lemma \ref{l1.3}. Hence the
result.

Let now $G$ be of order $64$. By \cite[Lemma 5.5]{CHKK}, find the
generators and relations of $G$. We will discuss only the case
$G=G(149)$ and leave the other groups to the reader. When
$G=G(149)$, take the abelian subgroup $H= \langle f_2, f_5 \rangle$.
Note that $Z(G)=\langle f_2^4, f_5 \rangle$. Construct two linear
characters of $H$, $\chi_1$ and $\chi_2$, by $\chi_1(f_2)=\zeta_8,
\chi_1(f_5)=1$ and $\chi_2(f_2)=1, \chi_2(f_5)=-1$. Let $\chi$ be
the direct sum of $\chi_1$ and $\chi_2$. The induced representation
is of degree $8$. The rest of the proof is the same as above.

\end{proof}

\newpage
%----------------------------------------References
\renewcommand{\refname}{\centering{References}}

\end{document}